\documentclass[12pt]{amsart}
\usepackage{txfonts, color,ulem}
\usepackage{mathrsfs}
\usepackage{amsmath}
\usepackage{amssymb, amsmath}
\usepackage{graphicx}

\newcommand{\newcom}{\newcommand}
\allowdisplaybreaks[3]
\newcom{\al}{\alpha}
\newcom{\be}{\beta}
\newcom{\eps}{\epsilon}
\newcom{\veps}{\varepsilon}
\newcom{\e}{\varepsilon}
\newcom{\ga}{\gamma}
\newcom{\Ga}{\Gamma}
\newcom{\ka}{\kappa}
\newcom{\Lam}{\Lambda}
\newcom{\lam}{\lambda}
\newcom{\Om}{\Omega}
\newcom{\om}{\omega}
\newcom{\Si}{\Sigma}
\newcom{\si}{\sigma}
\newcom{\tht}{\theta}
\newcom{\dtri}{\nabla}
\newcom{\tri}{\triangle}
\newcom{\oo}{\infty}
\newcom{\vphi}{\varphi}
\newcom{\nna}{\big\langle \nabla' \big\rangle}
\newcom{\aaa}{\mathfrak a}
\newcom{\aab}{\bar{\mathfrak a}}
\newcom{\aat}{\tilde{\mathfrak a}}
\newcom{\cB}{{\mathcal B}}
\newcom{\cC}{{\mathcal C}}
\newcom{\cD}{{\mathcal D}}
\newcom{\cF}{{\mathcal F}}
\newcom{\cH}{{\mathcal H}}
\newcom{\cL}{{\mathcal L}}
\newcom{\cM}{{\mathcal M}}
\newcom{\cN}{{\mathcal N}}
\newcom{\cP}{{\mathcal P}}
\newcom{\cS}{{\mathcal S}}
\newcom{\cQ}{{\mathcal Q}}
\newcom{\cT}{{\mathcal T}}
\newcom{\cY}{{\mathcal Y}}
\newcom{\cZ}{{\mathcal Z}}
\newcom{\R}{\mathbb R}
\newcom{\T}{\mathbb T}
\newcom{\bbT}{{\mathbb{T}}}
\newcom{\BT}{{\mathbb{T}}^2}
\newcom{\Z}{\mathbb Z}
\newcom{\C}{\mathbb C}
\newcom{\E}{\mathbb E}

\newcom{\hha}{\hat{\mathbf h}}
\newcom{\ha}{\hat{h}}
\newcom{\ul}{\underline}
\newcommand{\vc}[1]{{\mathbf #1}}
\newcom{\ve}{\vc{e}}
\newcom{\vN}{\vc{N}}
\newcom{\vn}{\vc{n}}
\newcom{\vG}{\vc{G}}
\newcom{\vF}{\vc{F}}
\newcom{\vf}{\vc{f}}
\newcom{\vg}{\vc{g}}
\newcom{\vq}{\vc{q}}
\newcom{\vu}{\vc{u}}
\newcom{\vv}{\vc{v}}
\newcom{\vw}{\vc{w}}
\newcom{\vb}{\vc{b}}
\newcom{\vh}{\vc{h}}
\newcom{\vz}{\vc{z}}
\newcom{\vup}{\vu^{+}}
\newcom{\vum}{\vu^{-}}
\newcom{\vvp}{\vv^{+}}
\newcom{\vvm}{\vv^{-}}
\newcom{\vbp}{\vb^{+}}
\newcom{\vbm}{\vb^{-}}
\newcom{\vhp}{\vh^{+}}
\newcom{\vhm}{\vh^{-}}
\newcom{\Omp}{{\Om^+}}
\newcom{\Omm}{{\Om^-}}
\newcom{\vupm}{{\vu^{\pm}}}
\newcom{\vvpm}{{\vv^{\pm}}}
\newcom{\vbpm}{{\vb^{\pm}}}
\newcom{\vhpm}{{\vh^{\pm}}}
\newcom{\vwp}{{\vc{w}^+}}
\newcom{\vwm}{{\vc{w}^-}}
\newcom{\vwpm}{{\vc{w}^{\pm}}}
\newcom{\Ompm}{{\Omega^{\pm}}}
\newcom{\vom}{\boldsymbol{\omega}}
\newcom{\vvap}{\vc{\varpi}}
\newcom{\vop}{\vom^{+}}
\newcom{\vnu}{\vc{\nu}}
\newcom{\vopm}{\vom^{\pm}}
\newcom{\vjp}{\vj^+}
\newcom{\vjm}{\vj^-}
\newcom{\vjpm}{\vj^{\pm}}
\newcom{\vj}{\boldsymbol{\xi}}
\newcom{\Ds} {\langle\nabla\rangle^{s-\f12}}

\newcom{\ds}{{\rm d} s}

\newcom{\f}{\frac}
\newcom{\di}{\displaystyle\int}
\newcom{\dl}{\displaystyle\lim}
\newcom{\ov}{\overline}
\newcom{\sset}{\subset}
\newcom{\wt}{\widetilde}
\newcom{\pa}{\partial}
\newcom{\p}{\partial}
\newcom\na{\nabla}
\newcom{\suml}{\sum\limits}
\newcom{\supl}{\sup\limits}
\newcom{\intl}{\int\limits}
\newcom{\infl}{\inf\limits}
\newcom{\disp}{\displaystyle}
\newcom{\non}{\nonumber}
\newcom{\no}{\noindent}
\newcom{\QED}{$\square$}

\def\div{\mathop{\rm div}\nolimits}
\def\curl{\mathop{\rm curl}\nolimits}

\def\eqdefa{\buildrel\hbox{\footnotesize def}\over =}

\newtheorem{athm}{\bf \t}[section]
\newenvironment{thm} [1] {\def\t{#1}\begin{athm} \bf \rm} {\end{athm}}
\newcom{\bthm}{\begin{thm}}\newcom{\ethm}{\end{thm}}

\newtheorem{theorem}{Theorem}[section]
\newtheorem{lemma}{Lemma}[section]

\newtheorem{corollary}{Corollary}[section]

\newtheorem{proposition}{Proposition}[section]

\newcom{\beq}{\begin{equation}}
\newcom{\eeq}{\end{equation}}
\newcom{\ben}{\begin{eqnarray}}
\newcom{\een}{\end{eqnarray}}
\newcom{\beno}{\begin{eqnarray*}}
\newcom{\eeno}{\end{eqnarray*}}
\newcom{\bali}{\begin{aligned}}
\newcom{\eali}{\end{aligned}}

\numberwithin{equation}{section}

\begin{document}

\title[Well-posedness of the free boundary problem in elasticdynamic]{Well-posedness of the free boundary problem in elasticdynamic with mixed stability condition}

\author{Hui Li}
\address{Department of Mathematics, Zhejiang University, Hangzhou 310027, China}
\email{lihui92@zju.edu.cn}

\author{Wei Wang}
\address{Department of Mathematics, Zhejiang University, Hangzhou 310027, China}
\email{wangw07@zju.edu.cn}

\author{Zhifei Zhang}
\address{School of  Mathematical Sciences, Peking University, Beijing 100871, China}
\email{zfzhang@math.pku.edu.cn}

\begin{abstract}
	In this paper, we prove the local well-posedness of the free boundary problem in incompressible elastodynamics under a mixed type stability condition,
i.e., for each point of the free boundary, at least one of the Taylor sign condition $-\pa_np>0$ and the non-collinearity condition holds.
This gives an affirmative answer to a problem raised by Trakhinin in \cite{Tra3}.
\end{abstract}
\maketitle

\section{Introduction}
\subsection{Presentation of the problem.}
In this paper, we consider the incompressible inviscid flow in 3-D elastodynamics:
\begin{equation}\label{eq:elso}
	\left\{
	\begin{array}{l}
		\pa_t \vu +  \vu\cdot\nabla\vu+ \nabla p= \div(\vF\vF^\top), \\
		\pa_t  \vF +  \vu\cdot\nabla \vF = \nabla\vu \vF,\\
		\div \vu = 0, \\
	\end{array}
	\right.
\end{equation}
where $\vu(t,x)=(u_1,u_2,u_3)$ denotes the fluid velocity,
$p(t,x)$ is the pressure, $\vF(t,x)=(F_{ij})_{3\times 3}$ is the deformation tensor, $\vF^\top=(F_{ji})_{3\times 3}$ denotes the transpose of the matrix $\vF$,
$\vF\vF^\top$ is the Cauchy-Green tensor in the case of neo-Hookean elastic materials,
$(\nabla\vu)_{ij}=\pa_ju_i$, $(\nabla\vu \vF)_{ij}=\sum^3_{k=1}\vF_{kj}\partial_k u_i$,
$(\div \vF^\top)_i=\sum^3_{j=1}\pa_jF_{ji}$,
$(\div(\vF\vF^\top))_i=\sum^3_{j,k=1}\pa_j(F_{ik}F_{jk})$.\smallskip

We will study the solution of (\ref{eq:elso}) defined in a time-dependent domain. Precisely, we let
\begin{align*}
	&	\Om=\mathbb{T}^2\times[-1,1]\subset \mathbb{R}^3,\quad  \Gamma_f=\{x\in\Om|x_3=f(t,x'),x'=(x_1,x_2)\in\bbT^2\},\\
	& \Om_f=\{x\in\Om|x_3\le f(t,x'),x'=(x_1,x_2)\in\bbT^2\},\qquad Q_T=\underset{t\in(0,T)}{\bigcup}\{t\}\times\Om_f,
\end{align*}
where $\Gamma_f$ is the free boundary and is assumed to be a graph. The system reads as
\begin{equation}\label{eq:els}
	\left\{
	\begin{array}{ll}
		\pa_t \vu + \vu\cdot\nabla \vu+ \nabla p =  \sum\limits^3_{j=1}(\vF_j\cdot\nabla) \vF_j &\text{ in }\quad Q_T,  \\
		\div \vu = 0,\,\div \vF^\top = 0&\text{ in }\quad  Q_T,  \\
		\pa_t  \vF_j +  \vu\cdot\nabla \vF_j =(\vF_j\cdot\nabla)\vu&\text{ in }\quad Q_T,
	\end{array}
	\right.
\end{equation}
with the boundary conditions on the moving interface $\Gamma_f$:
\begin{align}\label{bc:els}
&	\vu\cdot\vN_f=\pa_tf,\qquad \vF_j\cdot\vN_f=0,\\
& \label{bc:elsp}
	p=0.
\end{align}
Here $\vF_j=(F_{1j},F_{2j},F_{3j})$, $\vN_f=(-\pa_1f,-\pa_2f,1)$ and $\vn_f= {\vN_f}/{|\vN_f|}$ is the outward unit normal vector. On the artificial boundary $\Gamma^-=\bbT^2\times\{-1\}$, we impose the following boundary conditions on $(\vu,\vF)$:
\begin{align}
 	u_3=0,\qquad F_{3j}=0\qquad \text{on}\ \Gamma^-.
\end{align}
The system \eqref{eq:els} is supplemented with the initial data
\begin{align}
	\vu(0,x)=\vu_0(x),\qquad \vF(0,x)=\vF_0(x)\qquad \text{in}\ \Omega_{f_0},
\end{align}
with
\begin{equation}\label{ini:els}
  \left\{
  	\begin{array}{ll}
  		\div \vu_0=0,\ \div\vF_{0j}=0&x\in\Omega_{f_0},\\
  		\vF_{0j}\cdot\vN_{f_0}=0&x\in\Gamma_{f_0},\\
  		u_{03}=0,\ F_{03j}=0&x\in\Gamma^-.
  	\end{array}
  \right.
\end{equation}
Let us remark that the divergence free restriction on $\vF_j$ is automatically satisfied if $\div\vF_{0j}=0$. Indeed, if we apply the divergence operator to the third equation of (\ref{eq:els}), we will deduce the following transport equation
\begin{align*}
	\pa_t\div\vF_j + \vu\cdot\nabla\div\vF_j=0.
\end{align*}
Similar argument can be also applied to yield that $\vF_j\cdot\vN_f = 0$ on $\Gamma_f$ if $\vF_{0j}\cdot\vN_{f_0} = 0$ on $\Gamma_{f_0}$.\smallskip

The main goal of this paper is to study the local well-posedness of the system \eqref{eq:els}-\eqref{ini:els} under some suitable stability conditions imposed on the initial data.

\subsection{Backgrounds}
The free boundary problems for incompressible inviscid flow have received a lot of attention in the past decades. It is well-known that, under the Taylor sign condition
\begin{align*}
   \vn\cdot\nabla p\le-\varepsilon<0\quad \text{ on }\,\,\Gamma_f,
\end{align*}
the water wave problem for the incompressible Euler flow is well-posed \cite{CL, Wu1, Wu2, SZ1, ZZ}. Otherwise, the system could be ill-posed \cite{Eb}, which is known as the Rayleigh-Taylor instability.
In addition, the vortex sheet problem for the incompressible  Euler flow is always ill-posed, which is called the Kelvin-Helmholtz instability\cite{Maj}.
However, the surface tension has been proved that it could stabilize the Kelvin-Helmholtz and Rayleigh-Taylor instability, see \cite{AM, CCS, SZ2}.

 Syrovatskij \cite{Sy,Ax} observed that the presence of strong tangential magnetic fields can stabilize the Kelvin-Helmholtz instability for magnetohydrodynamics.
 There are many important works devoted to the rigorous mathematical justification, see \cite{Tra1, Tra2, Chen, WY} for the compressible case
 and \cite{MTT1, Tra-in, CMST, SWZ} for the incompressible case.  We also refer to  \cite{MTT2, SWZ2,Tra-JDE, ST} for the plasma-vacuum problem in magnetohydrodynamics.
 The effect of the Taylor sign condition in the plasma-vacuum problem has been studied in \cite{HL, Hao, GW}.

There are also several progresses on the free boundary problems for inviscid elastodynamics. Chen-Hu-Wang \cite{CHW} analyzed the 2-D linearized stability and proved the stabilization effect of elasticity on compressible vortex sheets. Recently, Chen-Huang-Wang-Yuan \cite{CHWY} extended the results to the 3-D nonlinear compressible case. In \cite{Tra3}, Trakhinin proved the well-posedness of the fluid-vacuum free boundary problem in compressible elastodynamics under the condition that there are two columns of the $3\times 3$ deformation tensor which are non-collinear at each point of the initial surface. For the incompressible case, Hao-Wang \cite{HW} proved a priori estimates for solutions in Sobolev spaces under the Taylor sign condition. Li-Wang-Zhang \cite{LWZ} proved the stabilization effect of elasticity on both the vortex sheets and fluid-vacuum problem. Gu-Wang \cite{GW} proved the local well-posedness in a domain with two disconnected free boundaries,
where the Taylor sign condition and non-collinear condition hold on each free boundary.

The aim of this paper is to show the local well-posedness for the fluid-vacuum free boundary problem
in incompressible elastodynamics under a mixed type stability condition, i.e., for each point of the free boundary,
one of the Taylor sign condition and the non-collinear condition is satisfied. The most important contribution of this paper is that we derived a special evolution equation for the free interface, in which both effects of those two stability conditions can be reflected, and the combination of these two conditions will ensure this evolution equation to be strictly hyperbolic.

\subsection{Main result}
We define
\begin{align*}
	\Lambda(\vF)(x')\eqdefa\inf_{\phi_1^2+\phi_2^2=1}\sum_{j=1}^3(\underline{F_{1j}}\phi_1+\underline{F_{2j}}\phi_2)^2(x').
\end{align*}
If $\Lambda(\vF)(x')\ge c_0>0$, we say that $\vF$ is non-collinear at $x'$. Here $\underline{F}$ denotes the restriction of $F$ on $\Gamma_f$ and $x'$ denotes $(x_1,x_2)$ the first two components of $x$.

We assume that there exists $c_0>0$ such that
\begin{equation}\label{condition:s2}
  \left\{
    \begin{array}{ll}
      \vN_f\cdot\uline{\nabla p}\le-c_0<0&x'\in\Gamma^1,\\
      \Lambda(\vF)(x')>c_0&x'\in\Gamma^2,
    \end{array}
  \right.
\end{equation}
where $\Gamma^1$ and $\Gamma^2$ are open sets on $\bbT^2$ satisfying $\bbT^2/\Gamma^1\Subset\Gamma^2$. We call \eqref{condition:s2} the mixed type stability condition.

Let $D_{F_j}=\uline{F_{ij}}\partial_i'$. Our main result is stated as follows.
\begin{theorem}\label{thm:main}
	Let $s\ge4$ be an integer and assume that
  \begin{align*}
    D_{F_{0k}}f_0\in H^{s-1/2}(\bbT^2),\quad f_0\in H^s(\bbT^2),\quad\vu_0,\,\vF_0\in H^s(\Om_{f_0}).
  \end{align*}
    Furthermore, for two open sets $\Gamma^1,\Gamma^2\subset\bbT^2$ satisfying $\bbT^2/\Gamma^1\Subset\Gamma^2$, we assume that there exists $c_0>0$ so that
  \begin{itemize}
    \item[1.] $-(1-2c_0)\le f_0\le (1-2c_0)$;
    \item[2.] $-\vN_{f_0}\cdot\uline{\nabla p_0}\ge 2c_0$ on $\Gamma^1$;
    \item[3.] $\Lambda(\vF_0)(x')\ge 2c_0$ on $\Gamma^2$.
  \end{itemize}
  Then there exists $T>0$ such that the system (\ref{eq:els})-(\ref{ini:els}) admits a unique solution $(f, \vu, \vF)$ in $[0,T]$ satisfying
  \begin{itemize}
    \item[1.] $f\in L^\infty([0,T), H^{s}(\bbT^2))$, $D_{F_k}f\in L^\infty([0,T), H^{s- 1/2}(\bbT^2))$;
    \item[2.] $\vu,\,\vF\in L^\infty\big(0,T;H^{s}(\Omega_f)\big)$;
    \item[3.] $-(1-c_0)\le f\le (1-c_0)$;
    \item[4.]$-\vN_{f}\cdot\uline{\nabla p}\ge c_0$ on $\Gamma^1$;
    \item[5.] $\Lambda(\vF)(x')\ge c_0$ on $\Gamma^2$.
  \end{itemize}			
\end{theorem}

The rest of this paper is organized as follows. In Section 2, we introduce the harmonic coordinate, Dirichlet-Neumann operator and some basic estimates related to the theorem. In Section 3, we derive an evolution equation for the free interface in which  both effects of those two stability conditions can be reflected. In Section 4, we construct an $\e$-regularized system and construct an approximation sequence to the solution of the original system. In Section 5, we prove the existence and uniqueness of the solution.

\section{Harmonic coordinate and Dirichlet-Neumann operator}
In this section, we recall some facts and well-known results on the harmonic coordinate and Dirichlet-Neumann operators.

We first introduce some notations used throughout this paper.
We use $x=(x_1,x_2,x_3)$ or $y=(y_1,y_2,y_3)$ to denote the coordinates in the fluid region, and use $x'=(x_1,x_2)$  or $y'=(y_1,y_2)$ to denote the natural coordinates on the interface.
In addition, we will use the Einstein summation notation where a summation from 1 to 2 is implied over repeated index,
while a summation from 1 to 3 over repeated index will be explicitly figured out by the symbol $\sum$  (i.e. $a_ib_i=a_1b_1+a_2b_2, \sum_{i=1}^3a_ib_i=a_1b_1+a_2b_2+a_3b_3$).

For a function $g:\Omega_f\to\mathbb R$, we denote $\nabla g=(\pa_1g,\pa_2g,\pa_3g)$, and for a function $\eta:\bbT^2\to \mathbb R$, $\nabla'\eta=(\pa_1\eta,\pa_2\eta)$, it is the same for the operator $\Delta$ and $\Delta'$. For a function $g:\Omega_f\to\mathbb R$, we can define its trace on $\Gamma_f$, which are denoted by $\underline g(x')$. Thus, for $i=1,2$,
\begin{align*}
	\pa_i\underline g(x')=\pa_i g(x',f(x'))+\pa_3g(x',f(x'))\pa_if(x').
\end{align*}
In this paper we do not distinguish $D_t=\pa_t+u_1\pa_1+u_2\pa_2+u_3\pa_3$ on $\Omega_f$ and $D_t=\pa_t+\uline{u_1}\pa_1'+\uline{u_2}\pa_2'$ on $\Gamma_f$. Recalling that $\uline{\vu}\cdot\vN_f=\pa_tf$, for any function $v$ defined on $\Omega_f$, we have
\begin{align*}
	\uline {D_t v}=D_t \uline{v}.
\end{align*}
We denote by $||\cdot||_{H^s(\Omega_f)}$, $|\cdot|_{H^s}$ the Sobolev norm on $\Omega_f$ and $\bbT^2$ respectively.

In the free boundary problem, the functions $(\vu,\vF)$ are defined in a domain changing with time $t$.
To overcome this difficulty, we pull them back to the fixed reference domain $\Om_*$ \cite{LWZ}. Let $\Gamma_*$ be a fixed graph surface given by
\begin{align*}
	\Gamma_*=\{(y_1,y_2,y_3):y_3=f_*(y_1,y_2)\},
\end{align*}
where $f_*$ satisfies $\int_{\bbT^2}f_*(y')dy'=0$. The reference domain $\Omega_*$ is given by
\begin{align*}
	\Omega_*=\{y\in\Om|y_3\le f_*(y_1,y_2)\}.
\end{align*}
We will look for the free boundary which lies in a neighborhood of $f_*$. As a result, we define
\begin{align*}
	\Upsilon(\delta,k)&\eqdefa\Big\{f\in H^k(\mathbb{T}^2): |f-f_*|_{H^k(\mathbb{T}^2)}\le \delta \Big\}.
\end{align*}
For $f\in \Upsilon(\delta,k)$, we can define the graph $\Gamma_f$ by
\begin{align*}
	\Gamma_f\eqdefa\left\{x\in \Om_t| x_3=f(t,x'), \int_{\mathbb{T}^2}f(t,x')dx'=0 \right\}.
\end{align*}
Now we introduce the harmonic coordinate. Given $f\in\Upsilon(\delta,k)$, we define a map $\Phi_f$ from $\Omega_*$ to $\Omega_f$ by the harmonic extension:
\begin{equation}
	\left\{
	\begin{array}{ll}
		\Delta_y \Phi_f=0, &y\in \Omega_*,\\
		\Phi_f(y',f_*(y'))=(y',f(y')),  &y'\in\mathbb{T}^2,\\
		\Phi_f(y',-1)=(y',-1),  &y'\in\mathbb{T}^2.\\
	\end{array}
	\right.
\end{equation}
Given $\Gamma_*$, there exists $\delta_0=\delta_0(|f_*|_{W^{1,\infty}})>0$ so that $\Phi_f$ is a bijection when $\delta\le \delta_0$. Then we can define an inverse map $\Phi_f^{-1}:\Omega_f\to\Omega_*$ such that
\begin{equation}\nonumber
	\Phi_f^{-1}\circ\Phi_f=\Phi_f\circ\Phi_f^{-1}=\mathrm{Id}.
\end{equation}
The following properties come from \cite{SWZ}.
\begin{lemma}\label{lem:basic}
Let $f\in \Upsilon(\delta_0,s-\f12)$ for $s\ge 3$. Then there exists a constant $C$ depending only on $\delta_0$ and  $\|f_*\|_{H^{s-\f12}}$ so that

\begin{itemize}
	\item[1.] If $u\in H^{\sigma}(\Omega_f)$ for $\sigma\in [0,s]$, then
	\begin{align*}
		\|u\circ\Phi_f\|_{H^\sigma(\Omega_*)}\le C\|u\|_{H^\sigma(\Omega_f)}.		
	\end{align*}
	\item[2.] If $u\in H^{\sigma}(\Om_*)$ for $\sigma\in [0,s]$, then
	\begin{align*}
		\|u\circ\Phi_f^{-1}\|_{H^{\sigma}(\Omega_f)}\le C\|u\|_{H^\sigma(\Om_*)}.		
	\end{align*}
	\item[3.] If $u, v\in H^{\sigma}(\Om_*)$ for $\sigma\in [2,s]$, then
	\begin{align*}
		\|uv\|_{H^\sigma(\Omega_f)}\le C\|u\|_{H^\sigma(\Omega_f)}\|v\|_{H^\sigma(\Omega_f)}.
	\end{align*}
\end{itemize}
\end{lemma}
We will use the Dirichlet-Neumann operator, which maps the Dirichlet boundary value of a harmonic function to its Neumann boundary value. That is to say,
for any $g(x')=g(x_1,x_2)\in H^k(\mathbb{T}^2)$, we denote by $\mathcal{H}_f g$  the harmonic extension to $\Omega_f$:
\begin{equation}
	\left\{
	\begin{array}{ll}
		\Delta \mathcal{H}_f g =0,& x\in \Omega_f,\\
		(\mathcal{H}_f g)(x',f(x'))=g(x'),&  x'\in\mathbb{T}^2,\\
		\mathcal{H}_f g(x',-1)=0,&  x'\in\mathbb{T}^2.
	\end{array}
	\right.
\end{equation}
We define another harmonic extension for different use:
\begin{equation}
	\left\{
	\begin{array}{ll}
		\Delta \bar{\mathcal{H}}_f g =0,& x\in \Omega_f,\\
		(\bar{\mathcal{H}}_f g)(x',f(x'))=g(x'),&  x'\in\mathbb{T}^2,\\
		\partial_3\bar{\mathcal{H}}_f g(x',-1)=0,&  x'\in\mathbb{T}^2.
	\end{array}
	\right.
\end{equation}

Then, we define two kinds of Dirichlet-Neumann operators:
\begin{align}
  \mathcal{N}_fg\overset{def}{=}\vN_f\cdot(\nabla\mathcal{H}_fg)\big|_{\Gamma_f},\quad   \bar{\mathcal{N}}_fg\overset{def}{=}\vN_f\cdot(\nabla\bar{\mathcal{H}}_fg)\big|_{\Gamma_f}.
\end{align}
We will use the following properties from \cite{ABZ, SWZ}.
\begin{lemma}\label{lem:DN}
	It holds that
	\begin{itemize}
		\item[1.] $\mathcal{N}_f$ is a self-adjoint operator:
		\begin{align*}
			(\mathcal{N}_f\psi,\phi)=(\psi,\mathcal{N}_f\phi),\quad\forall \phi,\, \psi\in H^{1/2}(\bbT^2);
		\end{align*}
		\item[2.] $\mathcal{N}_f$ is a positive operator:
		\begin{align*}
			(\mathcal{N}_f\phi,\phi)=\|\na\mathcal{H}_f\phi\|_{L^2(\Omega_f)}^2\ge 0,\quad \forall \phi\in H^{1/2}(\bbT^2);	
		\end{align*}
		Especially, if $\int_{\bbT^2}\phi(x')dx'=0$, there exists $c>0$ depending on $c_0, \|f\|_{W^{1,\infty}}$ such that
		\begin{align*}
			(\mathcal{N}_f\phi,\phi)\ge c\|\mathcal{H}_f\phi\|_{H^1(\Omega_f)}^2\ge c|\phi|_{H^{1/2}}^2.
		\end{align*}
		\item[3.] $\mathcal{N}_f$ is a bijection from $H^{k+1}_0(\bbT^2)$ to $H^{k}_0(\bbT^2)$ for $k\ge 0$, where
		\begin{align*}
			H^{k}_0(\bbT^2)\eqdefa H^k(\bbT^2)\cap\Big\{\phi\in L^2(\bbT^2):\int_{\bbT^2}\phi(x')dx'=0\Big\}.
		\end{align*}
	\end{itemize}
\end{lemma}

\begin{lemma}\label{lem:DN-1}
If $f\in H^{s}(\bbT^2)$ for $s>3$, then it holds that for any $\sigma\in[\frac{1}{2},s]$,
\begin{align}
	|\mathcal{N}_f^{-1}\phi|_{H^\sigma}\le C(|f|_{H^s})|\phi|_{H^{\sigma-1}}.
\end{align}
\end{lemma}
\begin{proof}
	The proof can be found in \cite{SWZ}.
\end{proof}
The results in the above two lemmas also hold for $\bar{\mathcal{N}}_f$.

\section{Evolution equation for the free interface}
In this section, we derive the evolution equation for the interface from the original system (\ref{eq:els})-(\ref{ini:els}).
The key ingredient here is that the evolution equation for $\partial_i'f$ could reflect the stabilization mechanism of both two stability conditions explicitly.

Recalling the boundary conditions for $\vu$ and $\vF$,
we have
	\begin{align*}
		D_tf=\uline{u_3}.
	\end{align*}
For any function $g=g(t,x')$, we also have:
	\begin{align}
		[D_t, \partial_i']g=-\partial_i'\uline{u_j}\partial_j'g,\quad
		[D_t, \uline{F_{jk}}\partial_j']g=0,\quad
		[\partial_i', \uline{F_{jk}}\partial_j']g=\partial_i'\uline{F_{jk}}\partial_j'g.
	\end{align}
Then by $D_t^2f=D_t\uline{u_3}$, it follows that
\begin{align*}
	D^2_t\pa'_if&=D_t\pa_i'D_tf+D_t[D_t,\pa_i']f=\pa'_iD_t^2f+[D_t,\pa'_i]D_tf+D_t[D_t,\pa_i']f\\
	&=\pa_i'D_t\uline{u_3}-(\pa'_i\uline{u_j})\pa'_jD_tf-D_t((\pa_i'\uline{u_j})\pa_j'f)\\
	&=\pa_i'D_t\uline{u_3}-(\pa_i'\uline{u_j})\pa_j'D_tf-(\pa_i'\uline{u_j})D_t(\pa_j'f)-(D_t\pa_i'\uline{u_j})\pa_j'f\\
	&=\pa_i'D_t\uline{u_3}-(\pa_i'\uline{u_j})(\pa_j'\uline{u_k})\pa_k'f-2(\pa_i'\uline{u_j})D_t(\pa_j'f)-(\pa_i'D_t\uline{u_j})\pa_j'f+(\pa_i'\uline{u_k})(\pa_k'\uline{u_j})\pa_j'f\\
	&=\pa_i'D_t\uline{u_3}-2(\pa_i'\uline{u_j})D_t\pa_j'f-(\pa_i'D_t\uline{u_j})\pa_j'f
.\end{align*}
Since $\uline{\vF_j}\cdot \vN_f=0$, the first equation of \eqref{eq:els} implies that
\begin{align*}
	\uline{D_tu_i}&=-\uline{\partial_i p}+\sum^3_{j=1}\uline{\vF_j\cdot\nabla F_{ij}}\\
  &=-\uline{\pa_ip}+\sum^3_{j=1}\uline{F_{sj}}\pa_s'\uline{F_{ij}}+\sum^3_{j=1}\uline{\vF_j}\cdot\vN_f\uline{\pa_3F_{ij}}=-\uline{\pa_ip}+\sum^3_{j=1}\uline{F_{sj}}\pa_s'\uline{F_{ij}}.
\end{align*}
 Using the equality above, there holds
\begin{align*}
	&\partial_i'D_t\uline{u_3}-(\partial_i'D_t\uline{u_j})\partial_j'f\\
	=&-\partial_i'\uline{\partial_3p}+\sum^3_{k=1}\pa_i'(\uline{F_{sk}}\pa_s'\uline{F_{3k}})+(\pa_i'\uline{\pa_jp})\pa_j'f-\sum^3_{k=1}\big(\pa_i'(\uline{F_{sk}}\pa_s'\uline{F_{jk}})\big)\pa_j'f\\
  =&-\pa_i'\uline{\pa_3p}+(\partial_i'\uline{\pa_jp})\partial_j'f+\sum^3_{k=1}\Big((\partial_i'\uline{F_{sk}})\partial_s'(\uline{F_{jk}}\pa_j'f)+\uline{F_{sk}}\partial_s'\big((\partial_i'\uline{F_{jk}})\pa_j'f\big)+\uline{F_{sk}}\partial_s'(\uline{F_{jk}}\partial_i'\pa_j'f)\Big)\\
  &-\sum^3_{k=1}\Big((\partial_i'\uline{F_{sk}})(\partial_s'\uline{F_{jk}})\partial_j'f+\uline{F_{sk}}(\partial_i'\partial_s'\uline{F_{jk}})\partial_j'f\Big)\\
  =&-\pa_i'\uline{\pa_3p}+(\partial_i'\uline{\pa_jp})\partial_j'f+\sum^3_{k=1}\Big((\partial_i'\uline{F_{sk}})\uline{F_{jk}}\partial_s'\partial_j'f+\uline{F_{sk}}(\partial_i'\uline{F_{jk}})\partial_s'\partial_j'f+\uline{F_{sk}}\partial_s'(\uline{F_{jk}}\partial_i'\partial_j'f)\Big)\\
  =&-\pa_i'\uline{\pa_3p}+(\partial_i'\uline{\pa_jp})\partial_j'f+\sum^3_{k=1}\Big(2(\partial_i'\uline{F_{sk}})\uline{F_{jk}}\partial_s'\partial_j'f+\uline{F_{sk}}\partial_s'(\uline{F_{jk}}\partial_i'\partial_j'f)\Big),
\end{align*}
where we use $\uline{\vF_j}\cdot \vN_f=0$ again. Recalling that $D_{F_j}=\uline{F_{ij}}\partial_i'$, we arrive at
\begin{align*}
	D_t^2\partial_i' f=&-\partial_i'\uline{\partial_3p}+(\partial_i'\uline{\partial_jp})\partial_j'f+\sum^3_{k=1}\Big(2(\partial_i'\uline{F_{sk}})\uline{F_{jk}}\partial_s'\partial_j'f+\uline{F_{sk}}\partial_s'(\uline{F_{jk}}\partial_i'\partial_j'f)\Big)-2(\pa_i'\uline{u_j})D_t\pa_j'f\\
	=&-\partial_i'\uline{\partial_3p}+(\partial_i'\uline{\partial_jp})\partial_j'f+\sum^3_{k=1}D^2_{F_k}\partial_i'f+\sum^3_{k=1}2(\partial_i'\uline{F_{sk}})D_{F_k}\partial_s'f-2(\pa_i'\uline{u_j})D_t\pa_j'f.
\end{align*}
Now, we {focus on} $-\partial_i'\partial_3p+(\partial_i'\partial_jp)\partial_j'f$.  We write
\begin{align}
	p=p_{\vu, \vu}-\sum^3_{j=1}p_{\vF_j, \vF_j},
\end{align}
where $p_{\vu_1, \vu_2}$ is the solution of the following equation:
\begin{equation}\label{eq:p}
	\left\{
	\begin{array}{ll}
		\Delta p_{\vu_1, \vu_2}= -\mathrm{tr}(\nabla\vu_1\nabla\vu_2)
		&x\in\Omega_f,\\
		p_{\vu_1, \vu_2}=0&x\in\Gamma_f,\\
		\pa_3 p_{\vu_1, \vu_2}=0&x\in\Gamma^-.
	\end{array}\right.
\end{equation}
Define auxiliary functions $q_i\eqdefa\partial_ip+\partial_3p\cH_f(\partial_i'f)$, which satisfies $q_i|_{\Gamma_f}=\partial_ip+\partial_3p\partial_i'f=0$. Then, we have
\begin{align*}
	-\partial_i'\uline{\partial_3p}+(\partial_i'\uline{\partial_jp})\partial_j'f
	=&-\uline{\partial_i\partial_3p}-\uline{\partial_3^2p}\partial_i'f+(\uline{\partial_i\partial_jp})\partial_j'f+\uline{\partial_3\partial_jp}\partial_i'f\partial_j'f\\
	=&-\uline{\partial_3q_i}+\uline{\partial_3p\partial_3\cH_f(\partial_i'f)}+\partial_j'f(\uline{\partial_jq_i}-\uline{\partial_3p\partial_j\cH_f(\partial_i'f)})\\
	=&-\vN_f\cdot\uline{\nabla q_i}+\uline{\partial_3p}\vN_f\cdot\uline{\nabla\cH_f(\partial_i'f)}.
\end{align*}
Finally, we arrive at
\begin{align}\label{eq:els-evo}
	D_t^2\partial_i' f=&\uline{\partial_3p}\vN_f\cdot\uline{\nabla\cH_f(\partial_i'f)}+\sum^3_{k=1}D^2_{F_k}\partial_i'f\\
	&\quad+\sum^3_{k=1}2(\partial_i'\uline{F_{sk}})D_{F_k}\partial_s'f-2(\pa_i'\uline{u_j})D_t\pa_j'f-\vN_f\cdot\uline{\nabla q_i}.\nonumber
\end{align}
From the fact that $p=0$ on $\Gamma_f$, the Taylor sign condition gives
\begin{align*}
    -|\vN_f|^2\uline{\pa_3p}=-\vN_f\cdot\uline{\nabla p}\ge c>0\qquad \text{on}\ \Gamma^1.
\end{align*}
{One can see that the equation \eqref{eq:els-evo} explicitly shows the stabilization mechanism for both two stability conditions.}

As $\Gamma_f$ is a graph, $|\vN_f|_{L^\infty}<\infty$. In the rest of this paper, we write the Taylor sign condition as $-\uline{\pa_3p}\ge c$ just for convenience.

\section{$\varepsilon$-regularized system}
{For the original system (\ref{eq:els})-(\ref{ini:els}), we choose the initial data $(f_0,\vu_0,\vF_0)$ satisfies following conditions:
\begin{enumerate}
\item[C1.] $D_{F_{0k}}f_0\in H^{s-1/2}(\bbT^2)$, $f_0\in H^s(\bbT^2)$, $\vu_0,\ \vF_0\in H^s(\Om_{f_0})$, where $s\ge4$ is an integer;
\item[C2.] For two open set $\Gamma^1,\Gamma^2\subset\bbT^2$ which satisfies $\bbT^2/\Gamma^1\Subset\Gamma^2$, there exists $c_0>0$ so that:\begin{itemize}
   \item[1.]$-(1-2c_0)\le f_0\le (1-2c_0)$;
    \item[2.]$-\vN_{f_0}\cdot\uline{\nabla p_0}\ge 2c_0$ on $\Gamma^1$;
    \item[3.]$\Lambda(\vF_0)(x')\ge 2c_0$ on $\Gamma^2$.
  \end{itemize}
\end{enumerate}
In this section, we introduce a regularized system with a suitable initial data.

Consider the system \eqref{eq:els}-\eqref{ini:els} with the boundary condition for $p$ in \eqref{bc:elsp} replaced by
\begin{align}\label{bc:e-elsp}
  p=-\varepsilon \bar{\mathcal N}_f^{-1}\Delta' f \qquad \text{on}\ \Gamma_f.
\end{align}
This system requires $f$ to have a higher regularity, thus we equip it with the initial data $(f^\e_0,\vu^\e_0,\vF^\e_0)$ as defined below. Let
\begin{align}\label{ini:e-f}
  f^\e_0\eqdefa \eta_\e*f_0,
\end{align}
where $\eta_\varepsilon=\frac{1}{\varepsilon}\eta(\frac{x'}{\sqrt{\varepsilon}})$ is a mollifier. Apparently, it holds that $\e|f^\e_0|^2_{H^{s+1/2}}\le C |f_0|^2_{H^s}$ and $-(1- \frac{3}{2} c_0)\le f^\e_0\le(1-\frac{3}{2}c_0)$ when $\e$ small enough.

We define harmonic coordinate $\Phi_{f^\e_0}:\Omega_{f_0}\to\Omega_{f^\e_0}$ based on $f^\e_0$. Define
\begin{align}\label{ini:e-els}
	\vu^\e_0\eqdefa P_{f^\e_0}^{\div}\big(\vu_0\circ\Phi_{f^\e_0}^{-1}\big),\qquad \vF^\e_0\eqdefa \bar P_{f^\e_0}^{\div}\big(\vF\circ\Phi_{f^\e_0}^{-1}\big),
\end{align}
where $P_{f^\e_0}^{\div}$ and $\bar P_{f^\e_0}^{\div}$ are projection operators which map a vector field
$\Omega_{f^\e_0}$ to its divergence-free part. More precisely,
$P_{f^\e_0}^{\div}\vv=\vv-\nabla\phi$, $\bar P_{f^\e_0}^{\div}\vv=\vv-\nabla\bar\phi$ with
\begin{equation}\nonumber
  \left\{
  	\begin{array}{ll}
  		\Delta\phi=\div\vv\quad&\text{in}\quad \Omega_{f^\e_0},\\
		\phi=0\quad&\text{on}\quad \Gamma_{f^\e_0},\\
		\partial_3\phi=0\quad&\text{on}\quad \Gamma^-,
  	\end{array}
  \right.
\end{equation}
and
\begin{equation}\nonumber
  \left\{
    \begin{array}{ll}
      \Delta\bar\phi=\div\vv\quad&\text{in}\quad \Omega_{f^\e_0},\\
    \vN_{f^\e_0}\cdot\nabla\bar\phi=\vv\cdot(\pa_1'f_0-\pa_1'f^\e_0,\pa_2'f_0-\pa_2'f^\e_0,0)\quad&\text{on}\quad \Gamma_{f^\e_0},\\
    \partial_3\bar\phi=0\quad&\text{on}\quad \Gamma^-.
    \end{array}
  \right.
\end{equation}
Thus, we have $\uline{\vF^\e_{0j}}\cdot\vN^\e=0$ and $(\vu^\e_0,\vF^\e_0)$ satisfies \eqref{ini:els}.
Since $D_{\vF_{0k}}f_0\in H^{s-1/2}(\bbT^2)$, $(\vu^\e_0,\vF^\e_0)$ have the same regularity as $(\vu_0,\vF_0)$.
It is straightforward to show that $(\vu^\e_0\circ\Phi_{f^\e_0},\vF^\e_0\circ\Phi_{f^\e_0})$ converges to $(\vu_0,\vF_0)$ in $H^s(\Omega_{f_0})$ when $\e$ tends to 0 \cite{SWZ}. We define
\begin{align}
  M^s_0&= |f_0|^2_{H^s}+\sum^3_{k=1}|D_{\vF_{0,k}}f_0|^2_{H^{s- 1/2}}+||\vu_0||^2_{H^s(\Omega_{f_0})}+||\vF_0||^2_{H^s(\Omega_{f_0})},\label{eq:energy-ini}\\
  M^s_\e&= \e|f^\e_0|^2_{H^{s+1/2}}+|f^\e_0|^2_{H^{s-1/2}}+||\vu^\e_0||^2_{H^s(\Omega_{f^\e_0})}+||\vF^\e_0||^2_{H^s(\Omega_{f^\e_0})}.\label{eq:energy-ini-e}
\end{align}
We choose $\hat\e$ small enough such that for all $\e\le\hat\e$ there exists a constant $C$ independent of $\e$ satisfying
\begin{align}\label{eq:ini-energy}
  M^s_\e\le CM^s_0.
\end{align}
We call the system \eqref{eq:els}-\eqref{bc:els}, \eqref{bc:e-elsp} with initial data \eqref{ini:e-f}-\eqref{ini:e-els} {\it the $\e$-regularized system}.

For this regularized system, one can obtain (the derivation is similar to (\ref{eq:els-evo}) but more simple; one can also see \cite{LWZ} for the detailed derivation)
\begin{align}\label{eq:efb1}
	\partial_tf^\e =&\theta^\e,\\
\label{eq:efb2}
	\partial_t\theta^\e=&-2(\underline{u^\e_1}\pa_1'\theta^\e+\underline{u^\e_2}\pa_2'\theta^\e)-\underset{s,r=1,2}{\sum}\underline{u^\e_s}\underline{u^\e_r}\pa_s'\pa_r'f^\e\\\nonumber
	&-\vN_{f^\e}\cdot\underline{\nabla (p_{\vu^\e,\vu^\e}-\underset{j}{\sum}p_{\vF_j^\e,\vF_j^\e})}+\underset{j}{\sum}\underset{s,r=1,2}{\sum}\underline{F_{sj}^\e}\underline{F_{rj}^\e}\pa_s'\pa_r'f^\e+\varepsilon\Delta' f^\e.	
\end{align}
Therefore,  we can write
\begin{align*}
	D^2_tf^\e=\sum_j\sum_{s,r=1,2}\underline{F_{sj}^\e F_{rj}^\e}\pa_s'\pa_r'f^\e+\varepsilon\Delta' f^\e+L(\theta^\e,f^\e),
\end{align*}
where $L(\theta^\e,f^\e)$ represents lower order terms.

The following proposition is the main results of this section.
\begin{proposition}\label{pro:app-seq}
    Assume that $(f_0,\vu_0,\vF_0)$ satisfies C1 and C2, then there exists constants $(\bar\e,\bar T)$, such that for each $\e\in(0,\bar\e]$
    and $T\in[0,\bar T]$ the system \eqref{eq:els}-\eqref{bc:els}, \eqref{bc:e-elsp} with initial data \eqref{ini:e-f}-\eqref{ini:e-els} admits a unique solution $(f^\e,\vu^\e,\vF^\e)$ in $[0,T]$ satisfying
  \begin{itemize}
    \item[1.] $ \e|f^\e|^2_{H^{s+1/2}}+|f^\e|^2_{H^{s-1/2}}+||\vu^\e||^2_{H^s(\Omega_{f^\e})}+||\vF^\e||^2_{H^s(\Omega_{f^\e})}\le CM^s_0$;
    \item[3.] $-(1-\frac{1}{2}c_0)\le f^\e\le (1- \frac{1}{2}c_0)$;
    \item[4.]$-\vN_{f^\e}\cdot\uline{\nabla p^\e}\ge \frac{1}{2}c_0$ on $\Gamma^1$;
    \item[5.] $\Lambda(\vF^\e)(x')\ge \frac{1}{2}c_0$ on $\Gamma^2$.
  \end{itemize}
  Here the index $s$ is given in C1, and the constant $C$ is independent of $\e$ and $T$.
\end{proposition}

Before presenting the proof, we give some results which will be used. Firstly, from \cite[Theorem 1.2]{LWZ}, the $\e$-regularized system is locally well-posed. More precisely, it holds
\begin{proposition}\label{pro:reg-wp}
 Let  $s\ge4$ be an integer.  Assume that
\begin{align*}
	f^\e_0\in H^{s+ 1/2}(\bbT^2),\quad \vu^\e_0,\,\vF^\e_0\in H^s(\Om_{f^\e_0}).
\end{align*}
Furthermore, assume that there exists $\bar c_0\le0$ so that
\begin{align*}
  -(1-2\bar c_0)\le f^\e_0\le (1-2\bar c_0).
\end{align*}
Then there exists $T=T(\e, M^s_\e)>0$ such that, the system \eqref{eq:els}-\eqref{bc:els} and \eqref{bc:e-elsp} with initial data \eqref{ini:e-f}-\eqref{ini:e-els}
admits a unique solution $(f^\e, \vu^\e, \vF^\e)$ on $[0,T]$ satisfying
\begin{itemize}
	\item[1.] $f^\e\in L^\infty([0,T), H^{s+1/2}(\bbT^2))$;
	\item[2.] $\vu^\e,\,\vF^\e\in L^\infty\big(0,T;H^{s}(\Omega_{f^\e})\big)$;
	\item[3.] $-(1-\bar c_0)\le f^\e\le (1-\bar c_0)$.
\end{itemize}
Moreover,  for $t\in [0, T]$ it holds
\begin{align*}
	\e|f^\e|^2_{H^{s+1/2}}+|f^\e|^2_{H^{s-1/2}}+||\vu^\e||^2_{H^s(\Omega_{f^\e})}+||\vF^\e||^2_{H^s(\Omega_{f^\e})}\le 2M^s_\e.
\end{align*}
\end{proposition}
To extend the solution to a time interval independent of $\e$,  we have to derive a uniform (in $\e$) a priori estimate for the solution $(f^\e, \vu^\e, \vF^\e)$.
We will drop the superscript $\e$ of $(f^\e, \vu^\e, \vF^\e, p^\e)$ in the rest of this section for convenience.

The following estimates will be frequently used in this section. In these estimates, we assume $f\in C^\infty(\bbT^2)$ and $\vu\in C^\infty (\Omega_f)$.
The proofs of these lemmas can be found in Appendix. We also remark that the commutator estimates below are inspired from \cite{SZ1}.
\begin{lemma}\label{lem:com-f-ds}
For any function $a\in H^s(\bbT^2)$ with $s>2$, we have
\begin{equation}
	\big|[a, \langle\nabla'\rangle^s]f\big|_{L^2}\le C|a|_{H^{s}}|f|_{H^{s-1}}.
\end{equation}
Here $\nna^s$ is the $s$-order derivatives on $\bbT^2$ which defined as follows
\begin{align*}
	\widehat{\langle \nabla'\rangle^s f}(\mathbf{k}) = \left( 1 + |\mathbf{k}|^2 \right)^{s/2}\hat{f}(\mathbf{k}), \quad \mathbf{k}=(k_1,k_2), \quad k_1,k_2\in\mathbb{Z}.
\end{align*}
\end{lemma}
\begin{corollary}\label{cor:com-dt-ds}
		For $s>2$, one has
		\begin{align*}
			\big|[D_t,\langle\nabla'\rangle^s]f\big|_{L^2}\le C\|\vu\|_{H^{s+1/2}(\Omega_f)}|f|_{H^{s}}.
		\end{align*}	
\end{corollary}

\begin{lemma}\label{lem:com-t-DN}
	For any function $g\in H^{s+1}(\bbT^2)$ with $s\ge \frac{3}{2}$, it holds that
	\begin{align*}
		|[\bar{\mathcal N}_f,D_t]g|_{H^{s}}\lesssim |f|^3_{H^{s+1}}||\vu||_{H^{s+3/2}(\Omega_f)}|g|_{H^{s+1}}.
	\end{align*}
\end{lemma}

\begin{lemma}\label{lem:com-f-DN}
	For any functions $a\in H^{3/2}(\bbT^2)$, $g\in H^{1/2}(\bbT^2)$, it holds that
	\begin{align*}
		|[\cN_f,a]g|_{L^2}\lesssim |f|_{H^3}|a|_{H^{3/2}}|g|_{H^{1/2}}.	
	\end{align*}
\end{lemma}

\begin{lemma}\label{lem:DN-b}
	For any functions $a \in H^{3/2}(\bbT^2)$, $g\in H^{1/2}(\bbT^2)$, it holds that
	\begin{align*}
		\int_{\bbT^2} a(\cN_fg)gdx'\lesssim |f|_{H^3}|a|_{H^{3/2}}|g|^2_{H^{1/2}}.
	\end{align*}
\end{lemma}

\begin{lemma}\label{lem:com-t-DN-2}
For any function $a\in H^{5/2}(\bbT^2)$, $g\in H^{1/2}(\bbT^2)$, it holds that
	\begin{align*}
		\int_{\bbT^2}a([D_t,\cN_f]g)gdx'\le C ||a||_{H^{5/2}}|f|^2_{H^4}||\vu||_{H^{4}(\Omega_f)}|g|^2_{H^{1/2}}.
	\end{align*}
	
\end{lemma}

\begin{lemma}\cite{La}\label{lem:com-ds-DN}
  For any function $g\in H^{s-3/2}(\bbT^2)$ with $s\ge3$, 
  it holds that 
  \begin{align*}
    |[\nna^{s-3/2},\cN_f]g|_{L^2}\lesssim |f|^2_{H^{s-1/2}}|g|_{H^{s-3/2}}.
   \end{align*}
\end{lemma}

\subsection{Stability condition of the $\e$-regularized system}
This subsection is devoted to showing that the mixed type stability condition is valid for the initial data of the $\e$-regularized system and it can be preserved in a uniform time.
\begin{lemma}\label{lem:tay-con}
	Giving $(f_0,\vu_0,\vF_0)$ satisfies C1 and C2, we assume that $(f,\vu,\vF)$ is the solution of system \eqref{eq:els}-\eqref{bc:els}, \eqref{bc:e-elsp} with initial data $(f^\e_0,\vu^\e_0,\vF^\e_0)$ defined in \eqref{ini:e-f}-\eqref{ini:e-els}, and satisfies
	\begin{align}\label{eq:e-energy-control}
		\e|f|^2_{H^{s+1/2}}+|f|^2_{H^{s-1/2}}+||\vu||^2_{H^s(\Omega_{f})}+||\vF||^2_{H^s(\Omega_{f})}\le CM^s_0
	\end{align}
	on $[0,T]$. Then, there exist constants $\tilde\e$ and $\widetilde T\le T$ such that for each $\e\in(0,\tilde\e]$ and $t\in[0,\widetilde T]$,
  \begin{itemize}
    \item[1.] $-\uline{\pa_3\nabla p}(t,x')\ge \frac{1}{2}c_0$ on $\Gamma^1$;
    \item[2.] $\Lambda(\vF)(t,x')\ge \frac{1}{2}c_0$ on $\Gamma^2$.
  \end{itemize}
\end{lemma}
\begin{proof}
For the $\e$-regularized system, the pressure can be written as
\begin{align*}
	p=\bar p+\mathring p,
\end{align*}
where
\begin{align*}
	\bar p\eqdefa-\varepsilon\bar{\mathcal H}_{f}\bar{\mathcal N}_{f}^{-1}\Delta' {f},\qquad \mathring p\eqdefa p_{\vu,\vu}-\sum^3_{j=1}p_{\vF_j,\vF_j}.
\end{align*}
Recalling that $p_0=p_{\vu_0,\vu_0}-\sum^3_{j=1}p_{\vF_{0j},\vF_{0j}}$ satisfies the Taylor sign condition on $\Gamma_{f_0}$ with $x'\in\Gamma^1$, we choose $\e$ small enough such that
\begin{align*}
	-\uline{\pa_3(p_{\vu^\e_0,\vu^\e_0}-\sum^3_{j=1}p_{\vF^\e_{0j},\vF^\e_{0j}})}>\frac{3}{2}c_0\ \text{on}\ \Gamma^1,\quad \Lambda(\vF^\e_0)>\frac{3}{2}c_0\ \text{on}\ \Gamma^2.
\end{align*}
From Lemma \ref{lem:DN-1}, we know that
\begin{align*}
	|\uline{\pa_3\bar p}(0,\cdot)|_{L^\infty}\le ||\bar p(0,\cdot)||_{H^4(\Om_{f^2_0})}\le C\e |f^\e_0|_{H^{7/2}}|f^\e_0|_{H^{9/2}}\le \sqrt{\e}CM^s_{\e}.
\end{align*}
Taking $\e$ small enough, we have
\begin{align*}
	-\uline{\pa_3p}(0,x')>c_0 \quad \text{on}\ \Gamma^1.
\end{align*}

Similar to \cite{LWZ}, it is direct to show that the non-collinear condition $\Lambda(\vF)>\frac{1}{2}c_0$ will hold in a short time independent of $\e$. So, we only need to focus on the Taylor sign condition.

Firstly, we give an estimate of $\pa_t p$. Applying $D_t$ to the first equation of (\ref{eq:els}), with the help of the third equation of (\ref{eq:els}) we have
\begin{align*}
	D_t^2u_i+D_t\pa_i p&=\sum^3_{j=1}\big(D_t\vF_j\cdot\nabla F_{ij}+\vF_j\cdot \nabla D_t F_{ij}-(\vF_j\cdot\nabla\vu)\cdot\nabla F_{ij}\big)\\
&=\sum^3_{j=1}\vF_j\cdot \nabla (\vF_j\cdot\nabla u_i).
\end{align*}
Taking divergence on both side of the above equation, we have
\begin{align}\label{eq:p-evo}
	\partial_i^2D_tp=&\sum^3_{s=1}((\partial_i^2u_s)\pa_sp+(\partial_iu_s) \partial_i\pa_sp)\\
  \nonumber&-\sum^3_{k=1}\Big((\pa_iD_tu_k)\pa_ku_i+2(\pa_iu_k)\pa_k(D_tu_i)\Big)+\sum^3_{s,k=1}2(\pa_iu_s)(\pa_su_k)(\pa_ku_i)\\
	\nonumber&+\sum^3_{s,j,k=1}\Big((\partial_iF_{kj})(\pa_kF_{sj})\pa_su_i+F_{kj}(\pa_i\pa_kF_{sj})\pa_su_i+2(\pa_iF_{kj})F_{sj}\pa_k\pa_su_i\Big).
\end{align}
Here we used the result that $\div \vu=0$.

Recalling that
\begin{align*}
	\pa_3p=0\quad\text{on}\ \Gamma^-,\quad p=-\varepsilon \bar{\mathcal N}_f^{-1}\Delta' f \qquad \text{on}\ \Gamma_f,
\end{align*}
the equation \eqref{eq:p-evo} is  equipped with the following boundary conditions:
\begin{align*}
	D_tp&=-\varepsilon D_t\bar{\mathcal N}_f^{-1}\Delta' f=-\varepsilon\bar{\mathcal N}_f^{-1}D_t\Delta' f-\varepsilon\bar{\mathcal N}_f^{-1}[\bar{\mathcal N}_f,D_t]\bar{\mathcal N}_f^{-1}\Delta' f\quad&\text{on}\ \Gamma_f,\\
	\pa_3D_tp&=\pa_3u_1\pa_1p+\pa_3u_2\pa_2p\quad&\text{on}\ \Gamma^-.
\end{align*}
When $s\ge1$, we have the following estimate
\begin{align*}
	||D_tp||_{H^{s+2}(\Omega_f)}&\lesssim ||\Delta D_tp||_{H^{s}(\Omega_f)}+|f|_{H^{s+ 3/2}}|\uline{D_tp}|_{H^{s+3/2}(\bbT^2)}+|\pa_3D_tp|_{H^{s+1/2}(\Gamma^-)}.
\end{align*}
Substituting \eqref{eq:p-evo} into it, one can obtain that
\begin{align}\label{ineq:dtp}
&	||D_tp||_{H^{s+2}(\Omega_f)}\\
&\lesssim ||\vu||_{H^{s+3}(\Omega_f)}\big(||p||_{H^{s+3}(\Omega_f)}+||\vF||^2_{H^{s+3}(\Omega_f)}\big)
+||D_t\vu||_{H^{s+2}(\Omega_f)}||\vu||_{H^{s+2}(\Omega_f)}\nonumber\\
	&\quad+||\vu||^3_{H^{s+2}(\Omega_f)}+\varepsilon |f|^2_{H^{s+2}}|D_t\Delta' f|_{H^{s+1/2}}+\varepsilon|f|^2_{H^{s+2}}|[\bar{\mathcal N}_f,D_t]\bar{\mathcal N}^{-1}_f\Delta' f|_{H^{s+ 1/2}}\nonumber\\
	&\lesssim (1+||\vu||_{H^{s+3}(\Omega_f)}+||\vF||_{H^{s+3}(\Omega_f)}+||p||_{H^{s+3}(\Omega_f)}+|f|^4_{H^{s+5/2}}+\e|f|_{H^{s+7/2}})^3,\nonumber
\end{align}
where in the last inequality, we used Lemma \ref{lem:DN-1} and Lemma \ref{lem:com-t-DN}.

Recalling that
\begin{align*}
	p=-\varepsilon\bar{\mathcal H}_f\bar{\mathcal N}_f^{-1}\Delta' f+p_{\vu,\vu}-\sum^3_{j=1}p_{\vF_j,\vF_j},
\end{align*}
and using Lemma \ref{lem:DN-1}, we get
\begin{align}\label{ineq:p}
	||p||_{H^{s}(\Omega_f)}\lesssim||\vu||^2_{H^{s-1}(\Omega_f)}+||\vF||^2_{H^{s-1}(\Omega_f)}+\e |f|^2_{H^{s+1/2}}.
\end{align}
This means that $||D_tp||_{H^{3}(\Omega_f)}$ can be controlled by the initial energy $M^s_{\e}$. Using \eqref{eq:ini-energy},
when $\e\le\hat\e$, we have $|\uline{\pa_3D_tp}|_{L^\infty}\le (1+C{M^s_0}^2)^3$ and immediately
\begin{align*}
	|D_t\uline{\pa_3p}|_{L^\infty}\le \big(1+C{M^s_0}^2\big)^3.
\end{align*}
As a conclusion, it holds for any $x\in\Gamma_f$ that
\begin{align*}
	\pa_3p(t,x)>\pa_3p(0,x)-t(1+C{M^s_0}^2)^3.
\end{align*}
Then there exists a constant $\widetilde T$, which is independent of $\e$,  such that $-\uline{\pa_3p}(t,x')>\frac{1}{2}c_0$ for any $x'\in \Gamma_1$ and $t\in[0,\widetilde T]$.
This finishes the proof of Lemma \ref{lem:tay-con}.
\end{proof}
\subsection{$\e$-independent energy estimate}

Now we derive the $\e$-independent energy estimate for the regularized system.

Similar to \eqref{eq:els-evo}, considering that
\begin{align*}
  -\partial_i'\uline{\partial_3\bar p}+(\partial_i'\uline{\partial_j\bar p})\partial_j'f=&-\vN_f\cdot \partial_i'(\uline{\nabla\bar p})\\
  =&-\partial_i'(\vN_f\cdot\uline{\nabla\bar p})+\partial_i'(\vN_f)\cdot\uline{\nabla\bar p}\\
  =&\varepsilon\partial_i'\Delta' f- \partial_1'\partial_i'f \uline{\partial_1\bar p}-\partial_2'\partial_i'f \uline{\partial_2\bar p},
\end{align*}
we can derived for the $\e$-regularized system that
\begin{align}\label{eq:fe-evo}
	D_t^2\partial_i' f=&\uline{\partial_3p}\vN_f\cdot\uline{\nabla\cH_f(\partial_i'f)}+\sum^3_{k=1}D^2_{F_k}\partial_i'f+\varepsilon\Delta' \partial_i'f+\sum^3_{k=1}2(\partial_i'\uline{F_{sk}})D_{F_k}\partial_s'f\\
	&-2(\pa_i'\uline{u_j})D_t\pa_j'f-\vN_f\cdot\uline{\nabla q_i}- \partial_1'\partial_i'f \uline{\partial_1\bar p}-\partial_2'\partial_i'f \uline{\partial_2\bar p},\nonumber
\end{align}
where $q_i=\partial_i\mathring p+\partial_3\mathring p\cH_f(\partial_i'f)$. Motivated by the above equation, we introduce
\begin{align*}
	E_\varepsilon^s(t)=&|D_t\langle \nabla'\rangle^{s-3/2}\partial_i'f|^2_{L^2}+\sum^3_{k=1}|D_{F_k} \langle \nabla'\rangle^{s-3/2}\partial_i'f|^2_{L^2}+\e|\partial_i'f|^2_{H^{s-1/2}}\\
	&+\int_{\Omega_f}\tilde{\mathfrak a}\big(\nabla\cH_f(\nna^{s-3/2}\partial_i'f)\big)^2dx+|f|^2_{L^2}+|\pa_tf|^2_{L^2}+||\vu||^2_{H^{s}(\Omega_f)}+||\vF||^2_{H^{s}(\Omega_f)},
\end{align*}
where $\tilde{\mathfrak a}$ is a suitably chosen function satisfying
\begin{equation*}
  \left\{
  	\begin{array}{ll}
  		0<c_0\le\tilde{\mathfrak a}\le C&x\in\Omega_f,\\
  		\uline{\tilde{\mathfrak a}}=-\uline{\pa_3p}&x'\in\Gamma^1.
  	\end{array}
  \right.
\end{equation*}
Therefore we have
\begin{align*}
	C||\partial_i'f||^2_{\dot H^{1/2}}\le\int_{\Omega_f}\tilde{\mathfrak a}\big(\nabla\cH_f(\partial_i'f)\big)^2dx\le C||\partial_i'f||^2_{\dot H^{1/2}}.
\end{align*}

The function $\tilde{\mathfrak a}$ is constructed in the following way. Define $\mathfrak a=-\uline{\pa_3p}$.
Recalling the proof of Lemma \ref{lem:tay-con}, \eqref{ineq:dtp} and \eqref{ineq:p} shows that for the solutions constructed in Proposition \ref{pro:reg-wp}, $\uline{-\pa_3p}$ have a uniform bound when $(t,x')\in([0,\tilde T],\bbT^2)$. Therefore, we can choose constant $\tilde c=\tilde c(M^s_0)$ such that $\tilde c-\uline{\pa_3p}\ge c_0$ for all $x'\in\bbT^2$. Then we let $\bar {\mathfrak a}= \mathfrak a+\phi \tilde c $, where $\phi \in C^\infty(\bbT^2)$ is a cutoff function satisfying
\begin{equation}
  \left\{
  	\begin{array}{ll}
  		0\le\phi\le1&x'\in \bbT^2,\\
	  	\phi=0 & x'\in \bbT^2/\Gamma^2,\\
	  	\phi=1 & x'\in \bbT^2/\Gamma^1.
  	\end{array}
  \right.
\end{equation}
If the Taylor sign condition holds on $\Gamma^1$, it follows that $\bar {\mathfrak a}\ge c_0 $ for $(t,x')\in [0,\tilde T]\times \bbT^2$. We choose $\tilde{\mathfrak a}$ as the solution of following equation:
\begin{equation}
  \left\{
  	\begin{array}{ll}
  		\Delta \tilde{\mathfrak a}=0&x\in\Omega_f,\\
  		\tilde{\mathfrak a}=\bar {\mathfrak a}&x\in\Ga_f,\\
  		\tilde{\mathfrak a}=c_0&x\in\Ga^-.
  	\end{array}
  \right.
\end{equation}
The maximum principle yields that
\begin{align*}
	c_0\le\tilde{\mathfrak a}\le\tilde c \qquad \text{on}\ \Omega_f.
\end{align*}
Thus, the constructed $\tilde{\mathfrak a}$ meets our requirement.

The uniform a priori estimate is stated as follows:
\begin{proposition}\label{pro:e-energy}
	Giving $(f_0,\vu_0,\vF_0)$ satisfies C1 and C2, we assume  that $(f,\vu,\vF)$ is the solution of system \eqref{eq:els}-\eqref{bc:els}, \eqref{bc:e-elsp} with initial data $(f^\e_0,\vu^\e_0,\vF^\e_0)$ defined in \eqref{ini:e-f}-\eqref{ini:e-els}, and satisfying stability condition \eqref{condition:s2} on $[0,T]$, then it holds that
	\begin{align*}
		\sup_{t\in[0,T]}E_\e^s(t)\le CE_\e^s(0)e^{CT},
	\end{align*}
where $C$ is a constant independent of $\e$.
\end{proposition}
We first introduce an elliptic estimate for a vector field.
\begin{lemma}\label{lem:ell-est} For any vector field $\vv$ defined on $\Omega_f$, we have
\begin{align*}
	||\vv||_{H^{s}(\Omega_f)}\le& C(|f|_{H^{s-1/2}})\Big(||\nabla\times\vv||_{H^{s-1}(\Omega_f)}+||\div\vv||_{H^{s-1}(\Omega_f)}\\
	&+\sum_{i=1,2}|\partial_i'\vv\cdot\vN_f|_{H^{s-3/2}}+||\vv||_{H^{s-1}(\Omega_f)}\Big).
\end{align*}
\end{lemma}
\begin{proof} The proof can be found in \cite{SZ2}. We represent it here for completeness. From the fact that
	\begin{align*}
		\Delta \vv=-\nabla\times(\nabla\times\vv)+\nabla(\div\vv),
	\end{align*}
	it suffices to prove (let $L_0=|f|_{H^{s-1/2}}$)
	\begin{align}\label{eq:elp-bd-est}
		|\vN_f\cdot\uline{\nabla\vv}|_{H^{s-3/2}}\le C(L_0)\big(||\nabla\times\vv||_{H^{s-1}(\Omega_f)}+||\div\vv||_{H^{s-1}(\Omega_f)}+\sum_{i=1,2}|\partial_i'\uline{\vv}\cdot\vN_f|_{H^{s-3/2}}\big).
	\end{align}
Define
	\begin{align*}
		\vw(\vv)=(\uline{\partial_1\vv}\cdot\vN_f, \uline{\partial_2\vv}\cdot\vN_f,\uline{\partial_3\vv}\cdot\vN_f).
	\end{align*}
Then \eqref{eq:elp-bd-est} is equivalent to
	\begin{align}\label{eq:esti-w}
		|\vw|_{H^{s-3/2}}\le C(L_0)\big(||\nabla\times\vv||_{H^{s-1}(\Omega_f)}+||\div\vv||_{H^{s-1}(\Omega_f)}+\sum_{i=1,2}|\partial_i'\uline{\vv}\cdot\vN_f|_{H^{s-3/2}}\big),
	\end{align}
	since  one has
	\begin{align*}
		|\vN_f\cdot\uline{\nabla\vv}-\vw|_{H^{s-3/2}}=|\vN_f\times(\uline{\nabla\times\vv})|_{H^{s-3/2}}\le C(L_0)(||\nabla\times\vv||_{H^{s-1}(\Omega_f)}).
	\end{align*}
	Direct calculations give that
	\begin{align}\label{eq:tangent}
		\tau_i\cdot\vw=\uline{\partial_i\vv}\cdot\vN_f+\uline{\partial_3\vv}\cdot\vN_f\partial_i'f=\partial_i'\uline{\vv}\cdot\vN_f,
	\end{align}
	where $\tau_1=(1,0,\partial_1f)$ and $\tau_2=(0,1,\partial_2f)$.

	To estimate $\vN_f\cdot\vw$, we write
	\begin{align}\label{eq:identity}
		(1+|\nabla' f|^2)\mathbf{I}=&(1+(\partial_2'f)^2)\tau_1\otimes\tau_1-\partial_1'f\partial_2'f(\tau_1\otimes\tau_2+\mathbf{\tau}_2\otimes\tau_1)\\
		&+(1+(\partial_1'f)^2)\tau_2\otimes\tau_2+\vN_f\otimes\vN_f.\nonumber		
	\end{align}
Since
	\begin{align*}
		\partial_i'\partial_j'\uline{\vv}=\partial_i'(\uline{\partial_j\vv}+\uline{\partial_3\vv}\partial_j'f)=\uline{\big(  (\tau_i\otimes\tau_j):\nabla^2\big)\vv}+\uline{\partial_3\vv}\partial_i'\partial_j'f,
	\end{align*}
one has
	\begin{align*}
		\uline{\big((\tau_i\otimes\tau_j):\nabla^2\big)\vv}\cdot\vN_f=&~\partial_i'\partial_j'\uline{\vv}\cdot\vN_f-\uline{\partial_3\vv}\cdot\vN_f\partial_i'\partial_j'f\\
		=&~\partial_i'(\partial_j'\uline{\vv}\cdot\vN_f)-\partial_j'\uline{\vv}\cdot\partial_i'\vN_f-\uline{\partial_3\vv}\cdot\vN_f\partial_i'\partial_j'f.
	\end{align*}
	Therefore, together with (\ref{eq:identity}), one can get
	\begin{align*}
		&\Big|\uline{\Big(\vN_f\otimes\vN_f:\nabla^2\Big)\vv}\cdot\vN_f-\uline{\Big((1+|\nabla' f|^2)\mathbf{I}:\nabla^2\Big)\vv}\cdot\vN_f\Big|_{H^{s-5/2}} \\
		&\le C(L_0)\big(|\partial_j'\uline{\vv}\cdot\vN_f|_{H^{s-3/2}}+||\vv||_{H^{s-1}(\Omega_f)}\big).
	\end{align*}
	On the other hand
	\begin{align*}
		\Big|\uline{\Big((1+|\nabla' f|^2)\mathbf{I}:\nabla^2\Big)\vv}\cdot\vN_f\Big|_{H^{s-5/2}}=&~\big|(1+|\nabla' f|^2)\uline{\Delta\vv}\cdot\vN_f\big|_{H^{s-5/2}}\nonumber\\
		\le&~ C(L_0)\big(||\nabla\times\vv||_{H^{s-1}(\Omega_f)}+||\div\vv||_{H^{s-1}(\Omega_f)}\big),
	\end{align*}
	which implies
	\begin{multline}\nonumber
		\Big|\uline{\Big(\vN_f\otimes\vN_f:\nabla^2\Big)\vv}\cdot\vN_f\Big|_{H^{s-5/2}}\le C(L_0)\Big(|\partial_j'\uline{\vv}\cdot\vN_f|_{H^{s-3/2}}+||\vv||_{H^{s-1}(\Omega_f)}\\
		+||\nabla\times\vv||_{H^{s-1}(\Omega_f)}+||\div\vv||_{H^{s-1}(\Omega_f)}\Big).
	\end{multline}
	We rewrite
	\begin{align}\nonumber
		\uline{\Big(\vN_f\otimes\vN_f:\nabla^2\Big)\vv}\cdot\vN_f=&~\vN_f\cdot\uline{\nabla(\vN_{f\cH}\otimes\vN_{f\cH}:\nabla\vv)}-\uline{(\vN_f\cdot\nabla)(\vN_{f\cH}\otimes\vN_{f\cH})}:\uline{\nabla\vv}\\
		\nonumber=&~\vN_f\cdot\uline{\nabla\cH_f(\vN_f\otimes\vN_f:\nabla\vv)}-\uline{(\vN_f\cdot\nabla)(\vN_{f\cH}\otimes\vN_{f\cH})}:\uline{\nabla\vv}\\
		\nonumber&+\vN_f\cdot\uline{\nabla\big(\vN_{f\cH}\otimes\vN_{f\cH}:\nabla\vv-
				\cH_f(\vN_f\otimes\vN_f:\nabla\vv)\big)},
	\end{align}
	where $\vN_{f\cH}$ denotes the harmonic extension of $\vN_f$, then we have
	\begin{align*}
		\big|\uline{(\vN_f\cdot\nabla)(\vN_{f\cH}\otimes\vN_{f\cH}):\nabla\vv}\big|_{H^{s-5/2}}\le C(L_0)||\vv||_{H^{s-1}(\Omega_f)}.
	\end{align*}
	For $Q=\vN_{f\cH}\otimes\vN_{f\cH}:\nabla\vv-\cH_f(\vN_f\otimes\vN_f:\nabla\vv)$, one has
	\begin{align*}
		Q|_{\Gamma_f}=Q|_{\Gamma_-}=0,\quad ||\Delta Q ||_{H^{s-3}(\Omega_f)}\le C(L_0)\Big(||\nabla\times \vv||_{H^{s-1}(\Omega_f)}+||\div \vv||_{H^{s-1}(\Omega_f)}\Big),
	\end{align*}
	which implies
	\begin{align*}
		|\vN_f\cdot\uline{\nabla\cH_f(\vN_f\otimes\vN_f:\nabla\vv)}|_{H^{s-5/2}}\le &C(L_0)\Big(|\partial_j'\uline{\vv}\cdot\vN_f|_{H^{s-3/2}}+||\vv||_{H^{s-1}(\Omega_f)}\\
		&+||\nabla\times\vv||_{H^{s-1}(\Omega_f)}+||\div\vv||_{H^{s-1}(\Omega_f)}\Big).		
	\end{align*}
	Noticing that $\vN_f\otimes\vN_f:\uline{\nabla\vv}=\vN_f\cdot\vw$, we have
	\begin{align*}
		|\vN_f\cdot\vw|_{H^{s-3/2}}\le &C(L_0)\Big(|\partial_j'\uline{\vv}\cdot\vN_f|_{H^{s-3/2}}+||\vv||_{H^{s-1}(\Omega_f)}\\
		&+||\nabla\times\vv||_{H^{s-1}(\Omega_f)}+||\div\vv||_{H^{s-1}(\Omega_f)}\Big).
	\end{align*}
	Together with (\ref{eq:tangent}), we obtain (\ref{eq:esti-w}), which finishes the proof of this lemma.
\end{proof}
Now, we are in the position to prove Proposition \ref{pro:e-energy}.
\begin{proof}
{\it Estimate of $f$.}
	Using \eqref{eq:fe-evo}, direct calculation shows that
	\begin{align*}
		&\frac{d}{dt}\frac{1}{2}|D_t \nna^{s-3/2}\partial_i'f|^2_{L^2}\\
		=&\int_{\bbT^2}(D_t\nna^{s-3/2} \partial_i'f)(D_t^2\nna^{s-3/2} \partial_i'f)-\frac{1}{2}\uline{u_j}\partial_j'(D_t\nna^{s-3/2} \partial_i'f)^2dx'\\
		=&\int_{\bbT^2}(D_t\nna^{s-3/2} \partial_i'f)(\nna^{s-3/2}D_t^2 \partial_i'f)dx'\\
		&+\int_{\bbT^2}(D_t\nna^{s-3/2} \partial_i'f)([D_t,\nna^{s-3/2}]D_t \partial_i'f+D_t[D_t,\nna^{s-3/2}] \partial_i'f)dx'\\
		&+\int_{\bbT^2}\frac{1}{2}\partial_j'\underline{u_j}(D_t\nna^{s-3/2} \partial_i'f)^2dx'\\
		=&\int_{\bbT^2}(D_t\nna^{s-3/2} \partial_i'f)(\nna^{s-3/2}D_t^2 \partial_i'f)dx'\\
		&+\int_{\bbT^2}(D_t\nna^{s-3/2} \partial_i'f)([D_t,\nna^{s-3/2}]D_t \partial_i'f+[\uline{u_j},\nna^{s-3/2}] D_t\partial_j'\partial_i'f)dx'\\
		&+\int_{\bbT^2}(D_t\nna^{s-3/2} \partial_i'f)([\uline{D_tu_j},\nna^{s-3/2}] \partial_j'\partial_i'f)+\frac{1}{2}\partial_j'\underline{u_j}(D_t\nna^{s-3/2} \partial_i'f)^2dx'\\
		\le& \int_{\bbT^2}(D_t\nna^{s-3/2} \partial_i'f)(\nna^{s-3/2}D_t^2 \partial_i'f)dx'\\
    &+C||\vu||^2_{H^{s-1}(\Omega_f)}|D_t \nna^{s-3/2}\partial_i'f|^2_{L^2}+C||D_t\vu||_{H^{s-1}(\Omega_f)}|D_t \nna^{s-3/2}\partial_i'f|_{L^2}|\partial_i'f|_{H^{s-3/2}}.
	\end{align*}
	From the first equation of \eqref{eq:els}, using \eqref{ineq:dtp} and \eqref{ineq:p}, we immediately have $||D_t\vu||_{H^{s-1}(\Omega_f)}\le C E^s_\e$. Next we consider the estimate of
	\begin{align*}
		&\int_{\bbT^2}(D_t\nna^{s-3/2} \partial_i'f)(\nna^{s-3/2}D_t^2 \partial_i'f)dx'\\
		=&\int_{\bbT^2}(D_t\nna^{s-3/2} \partial_i'f)(\nna^{s-3/2}\uline{\partial_3p}\vN_f\cdot\nabla\cH_f(\partial_i'f))dx'\\
		&+\int_{\bbT^2}(D_t\nna^{s-3/2} \partial_i'f)(\nna^{s-3/2}\sum^3_{k=1}D_{\vF_k}^2\partial_s'f)dx'\\
		&-\int_{\bbT^2}(D_t\nna^{s-3/2} \partial_i'f)\big(\nna^{s-3/2}(\vN_f\cdot\uline{\nabla q_i})\big)dx'\\
		&+\int_{\bbT^2}(D_t\nna^{s-3/2} \partial_i'f)\nna^{s-3/2}\big((\partial_i'\uline{F_{sk}})D_{\vF_k}\partial_s'f-2(\pa_i'\uline{u_j})D_t\pa_j'f\big)dx'\\
		&+\int_{\bbT^2}(D_t\nna^{s-3/2} \partial_i'f)\nna^{s-3/2}( \varepsilon\Delta' \partial_i'f- \partial_1'\partial_i'f \uline{\partial_1\bar p}-\partial_2'\partial_i'f \uline{\partial_2\bar p})dx'\\
		\triangleq&I_1+I_2+I_3+I_4+I_5.
	\end{align*}
\textit{ Estimate of $I_1$}.
Recalling the definition of $\aaa$, we rewrite
\begin{align*}
	I_1=&-\int_{\bbT^2}(D_t\nna^{s-3/2} \partial_i'f)\big(\nna^{s-3/2}\aaa\cN_f(\partial_i'f)\big)dx'\\
	=&-\int_{\bbT^2}\aaa(D_t\nna^{s-3/2} \partial_i'f)(\cN_f\nna^{s-3/2}(\partial_i'f))dx'\\
&-\int_{\bbT^2}\aaa(D_t\nna^{s-3/2} \partial_i'f)([\nna^{s-3/2},\cN_f](\partial_i'f))dx'\\
	&-\int_{\bbT^2}(D_t\nna^{s-3/2} \partial_i'f)([\nna^{s-3/2},\aaa]\cN_f(\partial_i'f))dx'.
\end{align*}
The first term is the principle one. To estimate it, we calculate
\begin{align*}
	&\frac{d}{dt}\int_{\bbT^2}\aaa(\cN_f\nna^{s-3/2}\partial_i'f)(\nna^{s-3/2}\partial_i'f)dx'\\
	=&\int_{\bbT^2}D_t \big[\aaa(\cN_f\nna^{s-3/2}\partial_i'f)(\nna^{s-3/2}\partial_i'f)\big]-\uline{u_j}\partial_j'\big[\aaa(\cN_f\nna^{s-3/2}\partial_i'f)(\nna^{s-3/2}\partial_i'f)\big]dx'\\
	=&\int_{\bbT^2}D_t(\aaa)(\cN_f\nna^{s-3/2}\partial_i'f)(\nna^{s-3/2}\partial_i'f)+\aaa  (D_t\cN_f\nna^{s-3/2}\partial_i'f)(\nna^{s-3/2}\partial_i'f)dx'\\
	&+\int_{\bbT^2}\aaa(\cN_f\nna^{s-3/2}\partial_i'f)D_t\nna^{s-3/2}\partial_i'fdx'+\int_{\bbT^2}(\partial_j'\uline{u_j})\aaa(\cN_f\nna^{s-3/2}\partial_i'f)(\nna^{s-3/2}\partial_i'f)dx'\\
  =&2\int_{\bbT^2}\aaa (D_t\nna^{s-3/2}\partial_i'f)(\cN_f\nna^{s-3/2}\partial_i'f)dx'\\
  &+\int_{\bbT^2}D_t(\nna^{s-3/2}\partial_i'f)([\cN_f,\aaa]\nna^{s-3/2}\partial_i'f)+\aaa([D_t,\cN_f]\nna^{s-3/2}\partial_i'f)(\nna^{s-3/2}\partial_i'f)dx'\\
  &+\int_{\bbT^2}D_t(\aaa)(\cN_f\nna^{s-3/2}\partial_i'f)(\nna^{s-3/2}\partial_i'f)+(\partial_j'\uline{u_j})\aaa(\cN_f\nna^{s-3/2}\partial_i'f)(\nna^{s-3/2}\partial_i'f)dx'.
\end{align*}
Where we use the following result. Since $\cN_f$ is a self-adjoint operator, we have
\begin{align*}
	&\int_{\bbT^2}\aaa  (D_t\cN_f\nna^{s-3/2}\partial_i'f)(\nna^{s-3/2}\partial_i'f)dx'\\
  =&\int_{\bbT^2}\aaa (D_t\nna^{s-3/2}\partial_i'f)(\cN_f\nna^{s-3/2}\partial_i'f)dx'+\int_{\bbT^2}(D_t\nna^{s-3/2}\partial_i'f)([\cN_f,\aaa]\nna^{s-3/2}\partial_i'f)dx'\\
	&+\int_{\bbT^2}\aaa([D_t,\cN_f]\nna^{s-3/2}\partial_i'f)(\nna^{s-3/2}\partial_i'f)dx'.
\end{align*}

We write
\begin{align*}
  \int_{\bbT^2}\aaa(\cN_f\nna^{s-3/2}\partial_i'f)(\nna^{s-3/2}\partial_i'f)dx'
  =&\int_{\bbT^2}\aab(\cN_f\nna^{s-3/2}\partial_i'f)(\nna^{s-3/2}\partial_i'f)dx'\\
  &-\int_{\bbT^2}(\aab-\mathfrak a)(\cN_f\nna^{s-3/2}\partial_i'f)(\nna^{s-3/2}\partial_i'f)dx'.
\end{align*}
Recalling the definition of $\aat$ and using integration by parts, one can get
\begin{align*}
  &\int_{\bbT^2}\aab(\nna^{s-3/2} \partial_i'f)(\cN_f\nna^{s-3/2} \partial_i'f)dx'\\
  =&\int_{\Omega_f}\tilde{\mathfrak a}(\nabla\cH_f(\nna^{s-3/2} \partial_i'f))^2dx +\frac{1}{2}\int_{\bbT^2}\vN_f\cdot \uline{\nabla \tilde{\mathfrak a}}(\nna^{s-3/2} \partial_i'f)^2dx'.
\end{align*}
Summarizing the above results gives that
\begin{align*}
	I_1=&-\frac{d}{dt}\frac{1}{2}\int_{\Omega_f}\tilde{\mathfrak a}\big(\nabla\cH(\nna^{s-3/2}\partial_i'f)\big)^2dx\\
	&-\frac{d}{dt}\frac{1}{4}\int_{\bbT^2}\vN_f\cdot \uline{\nabla \tilde{\mathfrak a}}(\nna^{s-3/2} \partial_i'f)^2dx'\\
  &+\frac{1}{2}\frac{d}{dt}\int_{\bbT^2}(\aab-\mathfrak a)(\cN_f\nna^{s-3/2}\partial_i'f)(\nna^{s-3/2}\partial_i'f)dx'\\
	&+\frac{1}{2}\int_{\bbT^2}D_t(\nna^{s-3/2}\partial_i'f)([\cN_f,\aaa]\nna^{s-3/2}\partial_i'f)dx'\\
	&+\frac{1}{2}\int_{\bbT^2}\aaa([D_t,\cN_f]\nna^{s-3/2}\partial_i'f)(\nna^{s-3/2}\partial_i'f)dx'\\
	&+\frac{1}{2}\int_{\bbT^2}D_t(\aaa)(\cN_f\nna^{s-3/2}\partial_i'f)(\nna^{s-3/2}\partial_i'f)dx'\\
	&+\frac{1}{2}\int_{\bbT^2}(\partial_j'\uline{u_j})\aaa(\cN_f\nna^{s-3/2}\partial_i'f)(\nna^{s-3/2}\partial_i'f)dx'\\
	&-\int_{\bbT^2}\aaa(\nna^{s-3/2}D_t \partial_i'f)([\nna^{s-3/2},\cN_f](\partial_i'f))dx'\\
  &-\int_{\bbT^2}(\nna^{s-3/2}D_t \partial_i'f)([\nna^{s-3/2},\aaa]\cN_f(\partial_i'f))dx'.
\end{align*}
We will control  all the terms on the right hand side except the first one by using the energy.

From Lemma \ref{lem:com-f-DN}, Lemma \ref{lem:com-t-DN-2} and Lemma \ref{lem:com-ds-DN} and the commutator estimates for the Dirichlet-Neumann operators,
we have \begin{align*}
	\int_{\bbT^2}D_t(\nna^{s-3/2}\partial_i'f)([\cN_f,\mathfrak a]\nna^{s-3/2}\partial_i'f)dx'
	\lesssim |D_t(\nna^{s-3/2}\partial_i'f|_{L^2}|f|_{H^3}|\aaa|_{H^{3/2}}|\partial_i'f|_{H^{s-1}},\\
	\int_{\bbT^2}\mathfrak a([D_t,\cN_f]\nna^{s-3/2}\partial_i'f)(\nna^{s-3/2}\partial_i'f)dx'
	\lesssim|\aaa|_{H^{5/2}}|f|^3_{H^{s-1/2}}||\vu||_{H^{4}(\Omega_f)}|\partial_i'f|^2_{H^{s-1}},\\
	\int_{\bbT^2}\aaa(D_t\nna^{s-3/2} \partial_i'f)([\nna^{s-3/2},\cN_f](\partial_i'f))dx'
	\lesssim|\aaa|_{L^\infty}|f|^2_{H^{s-1/2}}|D_t(\nna^{s-3/2}\partial_i'f|_{L^2}|\partial_i'f|_{H^{s}}.
\end{align*}
  Moreover, by Lemma \ref{lem:DN-b} we can get
\begin{align*}
	\int_{\bbT^2}D_t(\mathfrak a)(\cN_f\nna^{s-3/2}\partial_i'f)(\nna^{s-3/2}\partial_i'f)dx'&\lesssim |f|_{H^3}|D_t\aaa|_{H^{3/2}}|\partial_i'f|^2_{H^{s-1}},\\
	\int_{\bbT^2}(\partial_j'u_j)\mathfrak a(\cN_f\nna^{s-3/2}\partial_i'f)(\nna^{s-3/2}\partial_i'f)dx'&\lesssim|f|_{H^3}|\aaa|_{H^{3/2}}||\vu||_{H^{3}(\Omega_f)}|\partial_i'f|^2_{H^{s-1}}.
\end{align*}
Likewise, using Lemma \ref{lem:DN} and Lemma \ref{lem:com-f-ds}, one has
\begin{align*}
	\int_{\bbT^2}(D_t\nna^{s-3/2} \partial_i'f)([\nna^{s-3/2},\aaa]\cN_f(\partial_i'f))dx'\lesssim|D_t(\nna^{s-3/2}\partial_i'f|_{L^2}|\aaa|_{H^{s-3/2}}|\partial_i'f|_{H^{s-3/2}}.
\end{align*}
Recalling the definition of $\bar{\mathfrak a}$, $\bar{\mathfrak a}-\mathfrak a$ is a function independent of $t$ with support on $\Gamma^2$, we can use $|D_{F_k}\partial_i'f|^2_{H^{s-3/2}}$ to control
\begin{align*}
	&\frac{d}{dt}\frac{1}{2}\int_{\bbT^2}(\bar{\mathfrak a}-\mathfrak a)(\cN_f\nna^{s-3/2}\partial_i'f)(\nna^{s-3/2}\partial_i'f)dx'\\
	=&\frac{1}{2}\int_{\bbT^2}(\bar{\mathfrak a}-\mathfrak a)(D_t\cN_f\nna^{s-3/2}\partial_i'f)(\nna^{s-3/2}\partial_i'f)\\
	&\qquad\qquad+(\bar{\mathfrak a}-\mathfrak a)(\cN_f\nna^{s-3/2}\partial_i'f)(D_t\nna^{s-3/2}\partial_i'f)dx'\\
	&-\frac{1}{2}\int_{\bbT^2}(\bar{\mathfrak a}-\mathfrak a)\uline{u_j}\partial_j'\big((\cN_f\nna^{s-3/2}\partial_i'f)(\nna^{s-3/2}\partial_i'f)\big)dx'\\
	=&\frac{1}{2}\int_{\bbT^2}(\bar{\mathfrak a}-\mathfrak a)([D_t,\cN_f]\nna^{s-3/2}\partial_i'f)(\nna^{s-3/2}\partial_i'f)\\
	&\qquad\qquad+(D_t\nna^{s-3/2}\partial_i'f)\cN_f\big((\bar{\mathfrak a}-\mathfrak a)\nna^{s-3/2}\partial_i'f\big)dx'\\
	&+\frac{1}{2}\int_{\bbT^2}([(\bar{\mathfrak a}-\mathfrak a),\cN_f]\nna^{s-3/2}\partial_i'f)(D_t\nna^{s-3/2}\partial_i'f)\\
	&\qquad\qquad+\cN_f\big((\bar{\mathfrak a}-\mathfrak a)\nna^{s-3/2}\partial_i'f)(D_t\nna^{s-3/2}\partial_i'f\big)
	dx'\\
	&+\frac{1}{2}\int_{\bbT^2}\big(\partial_j'(\uline{u_j}(\bar{\mathfrak a}-\mathfrak a))\big)(\cN_f\nna^{s-3/2}\partial_i'f)(\nna^{s-3/2}\partial_i'f)dx'\\
	\lesssim&|\bar{\mathfrak a}-\mathfrak a|_{H^{3/2}}|f|_{H^4}||\vu||_{H^{4}(\Omega_f)}|\partial_i'f|^2_{H^{s-1}}+ |D_t\nna^{s-3/2}\partial_i'f|_{L^2}|(\bar{\mathfrak a}-\mathfrak a)\nna^{s-3/2}\partial_i'f|_{H^1}\\
	&+|\bar{\mathfrak a}-\mathfrak a|_{H^{3/2}}|f|_{H^3}|\partial_i'f|_{H^{s-1}}|D_t\nna^{s-3/2}\partial_i'f|_{L^2}+|\partial_j'(\uline{u_j}(\bar{\mathfrak a}-\mathfrak a))|_{H^{3/2}}|\partial_i'f|^2_{H^{s-1}}\\
	\lesssim&\big(1+|\bar{\mathfrak a}-\mathfrak a|_{H^{3/2}}+||\vu||_{H^{4}(\Omega_f)}+|f|_{H^4}\big)^3\big(|\partial_i'f|_{H^{s-1}}+|D_{F_k}\partial_i'f|_{H^{s-3/2}}+|D_t\nna^{s-3/2}\partial_i'f|_{L^2}\big)^2,
\end{align*}
where we have used Lemma \ref{lem:DN} and Lemma \ref{lem:com-f-DN}-\ref{lem:com-t-DN-2}. In a similar way, we have
\begin{align*}
	&\frac{d}{dt}\frac{1}{2}\int_{\bbT^2}\vN_f\cdot \uline{\nabla \tilde{\mathfrak a}}(\nna^{s-3/2} \partial_i'f)^2dx'\\
	=&\frac{1}{2}\int_{\bbT^2}D_t(\vN_f\cdot \uline{\nabla \tilde{\mathfrak a}})(\nna^{s-3/2} \partial_i'f)^2
	+2\vN_f\cdot \uline{\nabla \tilde{\mathfrak a}}(\nna^{s-3/2} \partial_i'f)(D_t\nna^{s-3/2} \partial_i'f)dx'\\
	&+\frac{1}{2} \int_{\bbT^2}(\partial_j'\uline{u_j})\vN\cdot \uline{\nabla \tilde{\mathfrak a}}(\nna^{s-3/2} \partial_i'f)^2dx'\\
	\lesssim&|D_t(\vN_f\cdot \uline{\nabla \tilde{\mathfrak a}})|_{L^2}|\nna^{s-3/2} \partial_i'f|^2_{L^4}\\
  &+|\vN_f\cdot \uline{\nabla \tilde{\mathfrak a}}|_{L^\infty}|D_t\nna^{s-3/2} \partial_i'f|_{L^2}|\partial_i'f|_{H^{s-3/2}}+|(\partial_j'\uline{u_j})\vN\cdot \uline{\nabla \tilde{\mathfrak a}}|_{L^\infty}|\partial_i'f|^2_{H^{s-3/2}}\\
	\lesssim&||\vu||^2_{H^{3}}|f|^2_{H^4}\big(|D_t\aab|_{H^{3/2}(\Omega_f)}+|\aab|_{H^{5/2}}\big)\big(|\partial_i'f|_{H^{s-1}}+|D_t\nna^{s-3/2} \partial_i'f|_{L^2}\big)^2.
\end{align*}
Recalling the definition of $\aab$ and $\aat$,  \eqref{ineq:dtp} and \eqref{ineq:p} ensure that all the norms related to these functions in our proof can be controlled by energy $E^s_\e(t)$.

Combining above estimates, it follows that
\begin{align*}
	I_1+\frac{d}{dt}\frac{1}{2}\int_{\Omega_f}\tilde{\mathfrak a}\big(\nabla\cH(\nna^{s-3/2}\partial_i'f)\big)^2dx\le P(E^s_\e(t)),
\end{align*}
where $P$ is a polynomial.

\vspace{10pt}

\textit{ Estimate of $I_2$}.
Using the fact that $[D_t,D_{\vF_k}]=0$, we have
\begin{align*}
	I_2=&\int_{\bbT^2}(D_t\nna^{s-3/2} \partial_i'f)(\nna^{s-3/2}\sum^3_{k=1}D_{\vF_k}^2\partial_s'f)dx'\\
	=&\sum^3_{k=1}\int_{\bbT^2}(D_t\nna^{s-3/2} \partial_i'f)([\nna^{s-3/2},D_{\vF_k}]D_{\vF_k}\partial_s'f+D_{\vF_k}[\nna^{s-3/2},D_{\vF_k}]\partial_s'f)dx'\\
	&+\sum^3_{k=1}\int_{\bbT^2}(D_t\nna^{s-3/2} \partial_i'f)(D_{\vF_k}D_{\vF_k}\nna^{s-3/2}\partial_s'f)dx'\\
	=&-\sum^3_{k=1}\int_{\bbT^2}D_t(D_{\vF_k}\nna^{s-3/2} \partial_i'f)(\nna^{s-3/2}D_{\vF_k}\partial_s'f)dx'\\
	&-\sum^3_{k=1}\int_{\bbT^2}(\partial_r'\uline{F_{rk}})(D_t\nna^{s-3/2} \partial_i'f)(D_{\vF_k}\nna^{s-3/2}\partial_i'f)dx'\\
	&\quad+\sum^3_{k=1}\int_{\bbT^2}(D_t\nna^{s-3/2} \partial_i'f)([\nna^{s-3/2},D_{\vF_k}]D_{\vF_k}\partial_s'f+D_{\vF_k}[\nna^{s-3/2},D_{\vF_k}]\partial_s'f)dx'\\
	\triangleq&I_{21}+I_{22}+I_{23}.
\end{align*}

Lemma \ref{lem:com-f-ds}  gives that
\begin{align*}
	I_{22}+I_{23}\lesssim&|D_{F_k}\nna^{s-3/2}\partial_i'f|^2_{L^2}|\vF|_{H^s(\Omega_f)}|D_t \nna^{s-3/2}\partial_i'f|^2_{L^2}.
\end{align*}
Since
\begin{align*}
	I_{21}+\frac{d}{dt}\frac{1}{2}|D_{F_k}\nna^{s-3/2}\partial_i'f|^2_{L^2}=&-\int_{\bbT^2}\uline{u_j} \partial_j'(\nna^{s-3/2}D_{F_k}\partial_i'f)^2dx'\\
	\le&|\nabla u|_{L^\infty(\Omega_f)}|D_{F_k}\nna^{s-3/2}\partial_i'f|^2_{L^2},
\end{align*}
we obtain the estimate of $I_2$:
\begin{align*}
	I_2+\frac{d}{dt}\frac{1}{2}|D_{F_k}\nna^{s-3/2}\partial_i'f|^2_{L^2}\lesssim P(E^s_\e(t)).
\end{align*}

\vspace{6pt}

\textit{ Estimate of $I_3$, $I_4$ and $I_5$}.
Recalling that $q_i=\partial_i\mathring p+\partial_3\mathring p\cH_f(\partial_i'f)$, $q_i|_{\Gamma_f}=\partial_i\mathring p+\partial_3\mathring p\partial_i'f=0$, from the definition of $\mathring p$ and $\cH_f$, we have
\begin{align*}
	\pa_3q_i|_{\Gamma^-}&=\big(\pa_i\pa_3\mathring p+\pa_{33}\mathring p\cH_f(\partial_i'f)+\partial_3\mathring p\partial_3\cH_f(\partial_i'f)\big)_{\Gamma^-}=0,\\
	\Delta q_i&=\partial_i\Delta \mathring p+\partial_3\Delta \mathring p\cH_f(\partial_i'f)+2\partial_3\nabla \mathring p\cdot\nabla\cH_f(\partial_i'f),
\end{align*}
which implies
\begin{align*}
	||q_i||_{H^{s}(\Omega_f)}&\le ||\partial_i\Delta \mathring p+\partial_3\Delta \mathring p\cH_f(\partial_i'f)+2\partial_3\nabla \mathring p\cdot\nabla\cH_f(\partial_i'f)||_{H^{s-2}(\Omega_f)}\\
	&\lesssim ||\mathring p||_{H^{s+2}(\Omega_f)}+||\mathring p||_{H^{s+1}(\Omega_f)}|\partial_i'f|_{H^{s-5/2}}+||\mathring p||_{H^{s+1}(\Omega_f)}|\partial_i'f|_{H^{s-3/2}}.
\end{align*}
Accordingly, it holds that
\begin{align*}
	I_3&=-\int_{\bbT^2}(D_t\nna^{s-3/2} \partial_i'f)\big(\nna^{s-3/2}(\vN_f\cdot\uline{\nabla q_i})\big)dx'\\
	&\le|D_t \nna^{s-3/2}\partial_i'f|_{L^2}||\vN\cdot\uline{\nabla q_i}|_{H^{s-3/2}}\le P(E^s_\e(t)).
\end{align*}
The term $I_4$ can be estimated in a similar way.  For $I_5$, we have
\begin{align*}
	I_5=\int_{\bbT^2}(D_t\nna^{s-3/2} \partial_i'f)\nna^{s-3/2}( \varepsilon\Delta' \partial_i'f- \partial_1'\partial_i'f \uline{\partial_1\bar p}-\partial_2'\partial_i'f \uline{\partial_2\bar p})dx'.
\end{align*}
Recalling the definition of $\bar p$, we have
\begin{align*}
	&\int_{\bbT^2}(D_t\nna^{s-3/2} \partial_i'f)\nna^{s-3/2} (\varepsilon\Delta' \partial_i'f)dx'\\
	=&\varepsilon\int_{\bbT^2}(D_t\nna^{s-3/2} \partial_i'f)\big(\Delta'\nna^{s-3/2} \partial_i'f\big)dx'\\
	=&-\varepsilon\int_{\bbT^2}\nabla'(D_t\nna^{s-3/2} \partial_i'f)\cdot\big(\nabla'\nna^{s-3/2}  \partial_i'f\big)dx'\\
	=&-\frac{\varepsilon}{2}\pa_t\int_{\bbT^2}(\nabla'\nna^{s-3/2} \partial_i'f)\cdot\big(\nabla'\nna^{s-3/2} \partial_i'f\big)dx'\\
&	-\varepsilon\int_{\bbT^2}([ \nabla', \uline{u_j}])\nna^{s-3/2}\partial_j'\partial_i'f\cdot\big(\nabla'\nna^{s-3/2} \partial_i'f\big)dx'\\
	&+\frac{\varepsilon}{2}\int_{\bbT^2}(\pa_j'\uline{u_j})(\nna^{s-3/2} \nabla'\partial_i'f)^2dx'.
\end{align*}
Using Lemma \ref{lem:DN-1}, we have
\begin{align*}
	\int_{\bbT^2}(D_t\nna^{s-3/2} \partial_i'f)\nna^{s-3/2}(\partial_j'\partial_i'f \partial_j\bar p)dx'
	\lesssim&|D_t \nna^{s-3/2}\partial_i'f|_{L^2}|\partial_j'\partial_i'f|_{H^{s-3/2}}|\uline{\partial_j\bar p}|_{H^{s-3/2}}\\
	\lesssim&|D_t \partial_i'f|_{H^{s-3/2}}|\partial_i'f|_{H^{s-1/2}}|\varepsilon \bar{\mathcal N}^{-1}\Delta' f|_{H^{s-3/2}}\\
	\lesssim&|D_t \partial_i'f|_{H^{s-3/2}}|f|_{H^{s-3/2}}\varepsilon|f|^2_{H^{s+1/2}}.
\end{align*}
Combining these estimates, we have
\begin{align*}
	I_5+\frac{d}{dt}\frac{1}{2}\varepsilon |\partial_i'f|^2_{H^{s+1}}\le P(E^s_\e(t)).
\end{align*}

{\bf Estimate of $(\vu,\vF)$}.
Recalling the boundary condition on $\Gamma_f$
\begin{align*}
	\vu\cdot\vN_f=\pa_t f,\quad\vF_j\cdot\vN_f=0,
\end{align*}
we have
\begin{align*}
	D_t\partial_i'f&~=\partial_i'D_tf-\partial_i'\uline{u_j}\partial_j'f=\partial_i'\uline{u_3}-\partial_i'\uline{u_j}\partial_j'f=\partial_i'\uline{\vu}\cdot\vN_f,\\
	\uline{F_{sk}}\partial_i'\partial_s'f&~=\partial_i'(\uline{\vF_k}\cdot\vN_f)-\uline{\vF_k}\cdot \partial_i'\vN_f=\partial_i'\uline{\vF_k}\cdot\vN_f.
\end{align*}
From Lemma \ref{lem:ell-est}, it holds that
\begin{align*}
	C||\vu||_{H^{s}(\Omega_f)}\le \big(||\nabla\times\vu||_{H^{s-1}(\Omega_f)}+\sum_{i=1,2}|D_t\partial_i'f|_{H^{s-3/2}(\bbT^2)}+||\vu||_{H^{s-1}(\Omega_f)}\big)\le CE^s_\e(t),\\
	C||\vF_k||_{H^{s}(\Omega_f)}\le\big(||\nabla\times\vF_k||_{H^{s-1}(\Omega_f)}+\sum_{i=1,2}|D_{\vF_k}\partial_i'f|_{H^{s-3/2}(\bbT^2)}+||\vF_k||_{H^{s-1}(\Omega)}\big)\le CE^s_\e(t).
\end{align*}
Then it suffices to prove  that
\begin{align*}
	||\nabla\times\vu||^2_{H^{s-1}(\Omega_f)}+||\vu||^2_{H^{s-1}(\Omega_f)}+||\nabla\times\vF_k||^2_{H^{s-1}(\Omega_f)}+||\vF||^2_{H^{s-1}(\Omega_f)}\le \int^T_0P(E^s_\e(t))dt.
\end{align*}
Similar to  \cite[Proposition 4.4]{LWZ}, it is direct to obtain that
\begin{align*}
	\frac{d}{dt}\Big(||\nabla\times\vu||^2_{H^{s-1}(\Omega_f)}+||\nabla\times\vF_k||^2_{H^{s-1}(\Omega_f)}\Big)\lesssim P(E^s_\e(t))(||\nabla\times\vu||^2_{H^{s-1}(\Omega_f)}+||\nabla\times\vF_k||^2_{H^{s-1}(\Omega_f)}).
\end{align*}
With the help of evolution equations of $\vu$ and $\vF$ in \eqref{eq:els}, we have
\begin{align*}
	&\frac{d}{dt}(||\nabla^{s-1}\vu||^2_{L^2(\Omega_f)}+\sum^3_{j=1}||\nabla^{s-1}\vF_j||^2_{L^2(\Omega_f)})\\
	\lesssim& \int_{\Omega_f}\na^{s-1}(\sum^3_{j=1}\vF_j\cdot\nabla \vF_j-\nabla p)\cdot \na^{s-1}\vu dx
	+\sum^3_{j=1}\int_{\Omega_f}\na^{s-1}(\vF_j\cdot\nabla \vu)\cdot \na^{s-1}\vF_j dx\\
	&\qquad+||\vu||^3_{H^{s-1}(\Omega_f)}+\sum^3_{j=1}||\vF_j||^2_{H^{s-1}(\Omega_f)}||\vu||_{H^{s-1}(\Omega_f)}\\
	\lesssim&-\sum^3_{j=1}\int_{\Omega_f}(\nabla\cdot\vF_j)\na^{s-1}\vu\cdot \na^{s-1}\vF_j dx\\
	&\qquad+||p||_{H^{s}(\Omega_f)}||\vu||_{H^{s-1}(\Omega_f)}+||\vu||^3_{H^{s-1}(\Omega_f)}+\sum^3_{j=1}||\vF_j||^2_{H^{s-1}(\Omega_f)}||\vu||_{H^{s-1}(\Omega_f)}\\
	\lesssim& \Big(\sum^3_{j=1}||\vF_j||_{H^{s-1}(\Omega_f)}+||\vu||_{H^{s-1}(\Omega_f)}+||p||_{H^{s}(\Omega_f)}\Big)^3.
\end{align*}
From \eqref{ineq:p} we have $||p||_{H^{s}(\Omega_f)}\lesssim E^s_\e(t)$, which implies
\begin{align*}
	\frac{d}{dt}(||\vu||_{H^{s}(\Omega_f)}+||\vF_j||_{H^{s}(\Omega_f)})\le P(E^s_\e(t)).
\end{align*}

{\bf Completing the proof}. Combining all these estimates above, we have
\begin{align*}
	\frac{d}{dt}E^s_\e(t)\le P(E^s_\e(t)).
\end{align*}
The proposition follows from Gronwall's inequality.
\end{proof}
\subsection{Approximation sequence to the solution of the original system}
Now, we give the proof of Proposition \ref{pro:app-seq}.
\begin{proof}
  Choose $\bar\e=\min(\hat\e,\tilde\e)$, where $\hat\e $ and $\tilde\e$ are given in \eqref{eq:ini-energy} and Lemma \ref{lem:tay-con} separately. For each $\e\le\bar \e$, we have
  \begin{align*}
    \e|f^\e_0|^2_{H^{s+1/2}}+|f^\e_0|^2_{H^{s-1/2}}+||\vu^\e_0||^2_{H^s(\Omega_{f^\e_0})}+||\vF^\e_0||^2_{H^s(\Omega_{f^\e_0})}\le C M^s_0.
  \end{align*}
  Using Proposition \ref{pro:reg-wp}, we solve $(f,\vu,\vF)$ on $[0,T(\e,(2C+\bar C)M^s_0)]$ which satisfies
  \begin{align*}
   \sup_{t\in[0,T]} \Big(\e|f|^2_{H^{s+1/2}}+|f|^2_{H^{s-1/2}}+||\vu||^2_{H^s(\Omega_{f})}+||\vF||^2_{H^s(\Omega_{f})}\Big)\le 2(2C+\bar C)M^s_0,
  \end{align*}
  where $\bar C$ is given in Proposition \ref{pro:e-energy}. Lemma \ref{lem:tay-con} shows that for $0\le t\le \hat T$, $(f,\vu,\vF)$ satisfies the mixed type stability condition \eqref{condition:s2}, where $\hat T=\min(\widetilde T,T)$. Proposition \ref{pro:e-energy} shows that there exists a constant $\bar C$ such that
  \begin{align*}
    \sup_{t\in[0,\hat T]}E_\e^s(t)\le \bar CE_\e^s(0)e^{\bar C\hat T}.
  \end{align*}
If $\bar CE_\e^s(0)e^{\bar C\hat T}\le(2C+\bar C)M^s_0$, then $(f,\vu,\vF)(\hat T)$ still satisfies the assumption of Proposition \ref{pro:reg-wp}.
Therefore, we can extend $(f,\vu,\vF)$ to $[0,\hat T+T]$. The above steps can be repeated until the energy is larger than $(2C+\bar C)M^s_0$.
Since $\bar C$ is independent of $\e$,  we can extend $(f,\vu,\vF)$ to a lifespan $\bar T$ independent of $\e$. This finishes the proof of Proposition \ref{pro:app-seq}.
\end{proof}
\section{Well-posedness of the original system}
\subsection{Taking the limit $\e\to0$}
To prove existences of solutions of the original system (\ref{eq:els})-(\ref{ini:els}), we consider the limit $\e\to0$.
For the solution $(f^\e,\vu^\e,\vF^\e)$ of each $\e$-regularized system given in Proposition \ref{pro:app-seq}, we pull $(\vu^\e,\vF^\e)$ back to a fixed domain.

We define
\begin{align*}
	\tilde\vu^\e=\vu^\e\circ\Phi_{f^\e},\quad \tilde\vF^\e=\vF^\e\circ\Phi_{f^\e},
\end{align*}
where $\Phi_{f^\e}:\Omega_{f_0}\to\Omega_{f^\e}$ is the harmonic coordinate defined in Section 2. Then by Lemma \ref{lem:basic}, the following energy
\begin{align*}
	\tilde E^s_\varepsilon(t)=&|D_t \nna^{s-3/2}\partial_i'f^\e|^2_{L^2}+\sum^3_{k=1}|D_{F^\e_k}\nna^{s-3/2} \partial_i'f^\e|^2_{L^2}+\varepsilon |\partial_i'f^\e|^2_{H^{s-1/2}}+|\partial_i'f^\e|^2_{H^{s-1}}\\
	&+|f^\e|^2_{L^2(\mathbb T^2)}+|\pa_tf^\e|^2_{L^2(\mathbb T^2)}+||\tilde\vu^\e||^2_{H^{s}(\Omega_{f_0})}+||\tilde\vF^\e||^2_{H^{s}(\Omega_{f_0})}
\end{align*}
also have an $\e$-independent upper bound on $[0,T]$.

Taking $\e\to0$, there exists a subsequence of $(f^\e,\tilde\vu^\e,\tilde\vF^\e)$ which converges weakly to a limit which we denote by $(f,\tilde\vu,\tilde\vF)$ satisfying
\begin{align*}
	D_tf\in L^2([0,T],H^{s-1/2}(\bbT^2)),&f\in L^2([0,T],H^{s}(\bbT^2)), D_{\tilde\vF_k}f\in L^2([0,T],H^{s-1/2}(\bbT^2)),\\
	\int_{\bbT^2}fdx'=0,\ \tilde\vu\in L^2([0,T],&H^{s}(\Omega_{f_0}),\tilde\vF\in L^2([0,T],H^{s}(\Omega_{f_0}).
\end{align*}
In addition, this convergence will be strong in spaces with low regularity. For
\begin{align*}
	\vu=\tilde\vu\circ\Phi_f^{-1},\quad \vF=\tilde\vF\circ\Phi_f^{-1},	
\end{align*}
standard arguments show that $(f,\vu,\vF)$ solves (\ref{eq:els})-(\ref{ini:els}).

Define the energy
\begin{align*}
	E^s(t)=&|D_t \nna^{s-3/2}\partial_i'f|^2_{L^2}+\sum^3_{k=1}|D_{F_k} \nna^{s-3/2}\partial_i'f|^2_{L^2}+\int_{\Omega_f}\tilde{\mathfrak a}\big(\nabla\cH_f(\nna^{s-3/2}\partial_i'f)\big)^2dx\\
	&+|f|^2_{L^2}+|\pa_tf|^2_{L^2}+||\vu||^2_{H^{s}(\Omega_f)}+||\vF||^2_{H^{s}(\Omega_f)}.
\end{align*}
Similar to Proposition \ref{pro:e-energy}, one can also have
\begin{align*}
	\sup_{t\in[0,T]}E^s(f,\vu,\vF)(t)\le2CM^s_0,
\end{align*}
where $T$ is the $\e$-independent lifespan obtained in Section 4.

\subsection{Uniqueness}
Assume that $(f^A,\vu^A,\vF^A)$ and $(f^B,\vu^B,\vF^B)$ are two solutions to (\ref{eq:els})-(\ref{ini:els}) on $[0,T]$. We denote the difference $(f^A-f^B,\tilde\vu^A-\tilde\vu^B,\tilde\vF^A-\tilde\vF^B)$ by $(f^D,\tilde\vu^D,\tilde\vF^D)$.

For the difference functions, we introduce the following energy
\begin{align*}
	E^s_D(t)=&|D_{t}^A \nna^{s-5/2}\partial_i'f^D|^2_{L^2}+\sum^3_{k=1}|D_{F^A_k} \nna^{s-5/2}\partial_i'f^D|^2_{L^2}+\int_{\Om_{f^A}}\tilde{\mathfrak a}^A\big(\nabla\cH_{f^A}(\nna^{s-5/2}\partial_i'f^D)\big)^2dx\\
	&+|f^D|^2_{L^2}+|\pa_tf^D|^2_{L^2}+||\tilde\vu^D||^2_{H^{s-1}(\Omega_{f_0})}+||\tilde\vF^D||^2_{H^{s-1}(\Omega_{f_0})},
\end{align*}
where $D_{t}^A$ means material derivation generated from $\vu^A$, and $\tilde{\mathfrak a}^A$ is defined by $(f^A,\vu^A,\vF^A)$ in the same way as $\tilde{\mathfrak a}$. Apparently $E^s_D(0)=0$. We will prove that $E^s_D(t)\equiv0$.

First of all, by elliptic estimates we know that
\begin{align*}
	||\Phi_{f^A}-\Phi_{f^B}||_{H^{s-1/2}(\Om_{f_0})}\lesssim |f^A-f^B|_{H^{s-1}}\lesssim E^s_D(t).
\end{align*}
Recall that
\begin{align}
	\pa_t^2\partial_i' f+\uline{u_j}\pa_t\pa_j' \partial_i'f+D^2_{u}\partial_i'f-\sum^3_{k=1}D^2_{F_k} \partial_i'f-\uline{\pa_3p}\cN_f(\partial_i'f)=g,
\end{align}
where $g$ is the lower order term
\begin{align*}
	g=-\pa_t\uline{u_j}\pa_j' \partial_i'f+\sum^3_{k=1}2(\partial_i'F_{sk})D_{\vF_k}\partial_s'f-2(\pa_i'u_j)D_t\pa_j'f+2(\pa_i'u_j)(\pa_j'u_s)\pa_s'f-\vN_f\cdot\nabla q_i.
\end{align*}
For the two evolution equations of $f^A$ and $f^B$, we use $g^A$ and $g^B$ to denote their lower order terms separately. Subtracting these two equations, we have
\begin{align}\nonumber
	&\pa_t^2\partial_i' (f^D)+\uline{u^A_j}\pa_t\pa_j' \partial_i'f^D+D^2_{u^A}\partial_i'f^D-\sum^3_{k=1}D^2_{F^A_k} \partial_i'f^D-\uline{\pa_3p^A}\cN_{f^A}(\partial_i'f^D)\\ \nonumber
	=&-(\uline{u^A_j}-\uline{u^B_j})\pa_t\pa_j'\pa_i'f^B-(D^2_{u^A}-D^2_{u^B})\pa_i'f^B+\sum^3_{k=1}(D^2_{F^A_k}-D^2_{F^B_k})\partial_i'f^B\\
	&+(\uline{\pa_3p^A}\cN_{f^A}\partial_i'f^B-\uline{\pa_3p^B}\cN_{f^B}\partial_i'f^B)+g^A-g^B. \label{eq:diff}
\end{align}
The {most difficult term} in (\ref{eq:diff}) is $\uline{\pa_3p^A}\cN_{f^A}\partial_i'f^B-\uline{\pa_3p^B}\cN_{f^B}\partial_i'f^B$, which can be written as
\begin{align*}
&	\uline{\pa_3p^A}\cN_{f^A}\partial_i'f^B-\uline{\pa_3p^B}\cN_{f^B}\partial_i'f^B\\
&=(\uline{\pa_3p^A}-\uline{\pa_3p^B})\vN_{f^A}\cH_{f^A}\partial_i'f^B
+(\uline{\pa_3p^B}(\vN_{f^A}-\vN_{f^B})\cH_{f^A}\partial_i'f^B)\\
&\quad+(\uline{\pa_3p^B}\vN_{f^B}(\cH_{f^A}-\cH_{f^B})\partial_i'f^B).
\end{align*}
{The first two terms can be estimated in standard methods, while the last term can be treat in the following way.}

For each function $\psi$ defined on $\bbT^2$, we define
\begin{align*}
	\widehat{\cH}_{f^B}\psi=\cH_{f^B}\psi\circ\Phi_{f^B}\circ\Phi_{f^A}^{-1}.
\end{align*}
Then we have
\begin{equation}
  \left\{
  	\begin{array}{ll}
  		\Delta(\cH_{f^A}\psi-\widehat{\cH}_{f^B}\psi)=-\Delta\widehat{\cH}_{f^B}\psi&x\in\Om_{f^A},\\
  		\cH_{f^A}\psi-\widehat{\cH}_{f^B}\psi(x',f^A(x'))=0&x'\in\bbT^2,\\
		\cH_{f^A}\psi-\widehat{\cH}_{f^B}\psi(x',-1)=0&x'\in\bbT^2.
  	\end{array}
  \right.
\end{equation}
By the definition of harmonic coordinate, we know that
\begin{align*}
	\pa_3^2&\widehat{\cH}_{f^B}\psi=(\pa_3^2\cH_{f^B}\psi)\circ\Phi_{f^B}\circ\Phi_{f^A}^{-1}\Big((\pa_3\Phi_{f^B}^{(3)})\circ\Phi_{f^A}^{-1}\frac{1}{(\pa_3\Phi_{f^A}^{(3)})\circ\Phi_{f^A}^{-1}}\Big)^2\\
	&+(\pa_3\cH_{f^B}\psi)\circ\Phi_{f^B}\circ\Phi_{f^A}^{-1}\frac{(\pa_3^2\Phi_{f^B}^{(3)})\circ\Phi_{f^A}^{-1}(\pa_3\Phi_{f^A}^{(3)})\circ\Phi_{f^A}^{-1}-(\pa_3^2\Phi_{f^A}^{(3)})\circ\Phi_{f^A}^{-1}(\pa_3\Phi_{f^B}^{(3)})\circ\Phi_{f^A}^{-1}}{((\pa_3\Phi_{f^A}^{(3)})\circ\Phi_{f^A}^{-1})^3}.
\end{align*}
The equations for $\pa_1^2\widehat{\cH}_{f^B}\psi$ and $\pa_2^2\widehat{\cH}_{f^B}\psi$ are similar. From $(\Delta\cH_{f^B}\psi)\circ\Phi_{f^B}\circ\Phi_{f^A}^{-1}=0$, we know that
\begin{align*}
	\Delta\widehat{\cH}_{f^B}\psi=C(\nabla (\Phi_{f^A}-\Phi_{f^B}),\nabla^2(\Phi_{f^A}-\Phi_{f^B}))\circ\Phi_{f^A}^{-1}.
\end{align*}
Then the right hand side in (\ref{eq:diff}) can be controlled by $E^s_D$.

Similar to Proposition \ref{pro:reg-wp}, we can prove
\begin{align*}
	\frac{d}{dt}\Big(|D_{t}^A \partial_i'f^D|^2_{H^{s-1}(\mathbb T^2)}+\sum^3_{k=1}|D_{F^A_k} \partial_i'f^D|^2_{H^{s-1}(\mathbb T^2)}+\int_{\Om_{f^A}}\tilde{\mathfrak a}^A\big(\nabla\cH_{f^A}(\nna^{s-1}\partial_i'f^D)\big)^2dx\Big)\lesssim E^s_D(t).
\end{align*}

Now we show that
\begin{align*}
	\frac{d}{dt}(||\tilde\vu^D||^2_{H^{s-1}(\Omega_{f_0})}+||\tilde\vF^D||^2_{H^{s-1}(\Omega_{f_0})})\lesssim E^s_D.
\end{align*}
For a vector field $\tilde\vv$ defined on $\Omega_{f_0}$, we define
\begin{align*}
	\curl_C \tilde \vv&=(\curl(\tilde\vv\circ\Phi_{f^C}^{-1}))\circ\Phi_{f^C},\\
	\div_C \tilde \vv&=(\div(\tilde\vv\circ\Phi_{f^C}^{-1}))\circ\Phi_{f^C}.
\end{align*}
It is clear that
 $$C||\vu^A-\tilde\vu^B\circ\Phi_{f^A}^{-1}||^2_{H^{s-1}(\Omega_{f^A})}\le||\tilde\vu^D||^2_{H^{s-1}(\Omega_{f_0})}\le C||\vu^A-\tilde\vu^B\circ\Phi_{f^A}^{-1}||^2_{H^{s-1}(\Omega_{f^A})}.$$
Thus, we only need to estimate $||\vu^A-\tilde\vu^B\circ\Phi_{f^A}^{-1}||^2_{H^{s+1/2}(\Omega_{f^A})}$. Recalling Lemma \ref{lem:ell-est}, we know that
\begin{align*}
	&||\vu^A-\tilde\vu^B\circ\Phi_{f^A}^{-1}||_{H^{s-1}(\Omega_{f^A})}\\
	\le& C(|f^A|_{H^{s-3/2}})\Big(||\curl\vu^A-\curl\tilde\vu^B\circ\Phi_{f^A}^{-1}||_{H^{s-2}(\Omega_{f^A})}+||\div\vu^A-\div\tilde\vu^B\circ\Phi_{f^A}^{-1}||_{H^{s-2}(\Omega_{f^A})}\\
	&+\sum_{i=1,2}|\partial_i'\uline{\vu^A}\cdot\vN_{f^A}-\partial_i'\uline{\tilde\vu^B\circ\Phi_{f^A}^{-1}}\cdot\vN_{f^A}|_{H^{s-5/2}(\bbT^2)}+||\vu^A-\tilde\vu^B\circ\Phi_{f^A}^{-1}||_{H^{s-2}(\Omega_{f^A})}\Big).
\end{align*}
For the terms on the right hand side, we have
\begin{align*}
	&||\curl\vu^A-\curl\tilde\vu^B\circ\Phi_{f^A}^{-1}||_{H^{s-2}(\Omega_{f^A})}\\
	\lesssim&||\curl_A\vu^A-\curl_B\vu^B||_{H^{s-2}(\Om_{f_0})}+||(\curl_B-\curl_A)\vu^B||_{H^{s-2}(\Om_{f_0})}\\
	\lesssim &||\curl_A\vu^A-\curl_B\vu^B||_{H^{s-2}(\Om_{f_0})}+||\Phi_{f^A}-\Phi_{f^B}||_{H^{s-1}(\Om_{f_0})},\\
	&||\div\vu^A-\div\tilde\vu^B\circ\Phi_{f^A}^{-1}||_{H^{s-2}(\Omega_{f^A})}\\
	\lesssim&||\div_A\vu^A-\div_B\vu^B||_{H^{s-2}(\Om_{f_0})}+||(\div_B-\div_A)\vu^B||_{H^{s-2}(\Om_{f_0})}\\
	\lesssim&||\Phi_{f^A}-\Phi_{f^B}||_{H^{s-1}(\Om_{f_0})},\\
	&|\partial_i'\uline{\vu^A}\cdot\vN_{f^A}-\partial_i'\uline{\tilde\vu^B\circ\Phi_{f^A}^{-1}}\cdot\vN_{f^A}|_{H^{s-5/2}(\bbT^2)}\\
	\lesssim&|\partial_i'\uline{\tilde\vu^A}\cdot\vN_{f^A}-\partial_i'\uline{\tilde\vu^B}\cdot\vN_{f^B}|_{H^{s-5/2}(\bbT^2)}+	|\partial_i'\uline{\tilde\vu^B}\cdot(\vN_{f^A}-\vN_{f^B})|_{H^{s-5/2}(\bbT^2)}\\
	\lesssim&|D_{t}^A\partial_i'f^D|_{H^{s-5/2}(\bbT^2)}+|D_{u^D}\partial_i'f^B|_{H^{s-5/2}(\bbT^2)}+|f^D|_{H^{s-3/2}(\bbT^2)},
\end{align*}
where $D_u=\uline{u_1}\pa_1'+\uline{u_2}\pa_2'$.

As a result, it holds that
\begin{align*}
	&||\tilde\vu^D||^2_{H^{s-1}(\Omega_{f_0})}\\
  \lesssim &|D_{t}^A \pa_if^D|_{H^{s-3/2}(\mathbb T^2)}+|f^D|_{H^{s-3/2}(\mathbb T^2)}+||\curl_A\vu^A-\curl_B\vu^B||_{H^{s-2}(\Om_{f_0})}+||\tilde\vu^D||^2_{H^{s-2}(\Omega_{f_0})}.
\end{align*}
We have similar estimates for $||\tilde\vF^D||^2_{H^{s-1}(\Omega_{f_0})}$.

Then using the method in Proposition \ref{pro:e-energy}, it follows that
\begin{align*}
	&\frac{d}{dt}\big(||\curl_A\vu^A-\curl_B\vu^B||_{H^{s-2}(\Om_{f_0})}+||\tilde\vu^D||^2_{H^{s-2}(\Omega_{f_0})}\\
	&\qquad\quad+\sum^3_{j=1}||\curl_A\vF_j^A-\curl_B\vF_j^B||_{H^{s-2}(\Om_{f_0})}+\sum^3_{j=1}||\tilde\vF_j^D||^2_{H^{s-2}(\Omega_{f_0})}\big)\lesssim E^s_D(t).
\end{align*}
Thus, finally we have
\begin{align*}
	\frac{d}{dt}E^s_D(t)\le C(M^s_0)E^s_D(t),
\end{align*}
which finishes the proof of uniqueness.

\begin{appendix}
\section{Estimates related to the D-N operator}

\begin{lemma}\label{lem-com-f-ds}
For any function $a\in H^s(\bbT^2)$ with $s>2$, we have
\begin{equation}
	\big|[a, \langle\nabla'\rangle^s]f\big|_{L^2}\le C|a|_{H^{s}}|f|_{H^{s-1}}.
\end{equation}
Here $\nna^s$ is the $s$-order derivatives on $\bbT^2$ which is defined as follows
\begin{align*}
	\widehat{\langle \nabla'\rangle^s f}(\mathbf{k}) = \left( 1 + |\mathbf{k}|^2 \right)^{\frac{s}{2}}\hat{f}(\mathbf{k}), \quad \mathbf{k}=(k_1,k_2), \quad k_1,k_2\in\mathbb{Z}.
\end{align*}
\end{lemma}

	\begin{corollary}\label{cor-com-dt-ds}
		For $s>2$, we have
		\begin{align*}
			\big|[D_t,\langle\nabla'\rangle^s]f\big|_{L^2}\le C\|\vu\|_{H^{s+1/2}(\Omega_f)}|f|_{H^{s}}.
		\end{align*}	
	\end{corollary}

\begin{lemma}\label{lem-basic-n}
Let $D_t=\pa_t+\uline{u_1}\pa_1'+\uline{u_2}\pa_2'$, then
	\begin{align*}
		D_t\vN_f=&\frac{-\pa_1'\underline{\vu}\cdot\vN_f-(\pa_2'f)^2\pa_1'\underline{\vu}\cdot\vN_f+\pa_1'f\pa_2'f\pa_2'\uline{\vu}\cdot\vN_f}{1+(\pa_1'f)^2+(\pa_2'f)^2}\tau_1\\
		&+\frac{-\pa_2'\underline{\vu}\cdot\vN_f-(\pa_1'f)^2\pa_2'\underline{\vu}\cdot\vN_f+\pa_1'f\pa_2'f\pa_1'\uline{\vu}\cdot\vN_f}{1+(\pa_1'f)^2+(\pa_2'f)^2}\tau_2\\
		&+\frac{\pa_1'f\pa_1'\underline{\vu}\cdot\vN_f+\pa_2'f\pa_2'\underline{\vu}\cdot\vN_f}{1+(\pa_1'f)^2+(\pa_2'f)^2}\vN_f,
	\end{align*}
	where
	\begin{align*}
		\vN_f=(-\pa_1'f,-\pa_2'f,1),\quad\tau_1=(1,0,\pa_1'f),\quad\tau_2=(0,1,\pa_2'f).		
	\end{align*}
\end{lemma}
\begin{proof}
Let $\vn_f=\frac{\vN_f}{|\vN_f|}$ and write $D_t\vN_f=A\tau_1+B\tau_2+(D_t\vN_f\cdot\vn_f)\vn_f$, where $A$, $B$ are undetermined coefficients. Direct calculation shows that
	\begin{align*}
	 	D_t\vN_f\cdot\vn_f=\frac{\pa_1'f\pa_1'\uline\vu\cdot\vN_f+\pa_2'f\pa_2'\uline\vu\cdot\vN_f}{\sqrt{1+(\pa_1'f)^2+(\pa_2'f)^2}}.
	\end{align*}
	Thanks to  $\vN_f\cdot\tau_1=\vN_f\cdot\tau_2=0$, it follows that
	\begin{align*}
		A(1+\pa_1'f^2)+B(\pa_1'f\pa_2'f)=D_t\vN_f\cdot\tau_1=-\pa_1'\uline\vu\cdot\vN_f.
	\end{align*}
	Similarly, we have
	\begin{align*}
		A(\pa_1'f\pa_2'f)+B(1+\pa_2'f^2)=D_t\vN_f\cdot\tau_2=-\pa_2'\uline\vu\cdot\vN_f.
	\end{align*}
From above equations, we can solve
	\begin{align*}
		A&=\frac{\big(-\pa_1'\underline{\vu}\cdot\vN_f-(\pa_2'f)^2\pa_1'\underline{\vu}\cdot\vN_f+\pa_1'f\pa_2'f\pa_2'\uline{\vu}\cdot\vN_f\big)}{1+(\pa_1'f)^2+(\pa_2'f)^2},\\
		B&=\frac{\big(-\pa_2'\underline{\vu}\cdot\vN_f-(\pa_1'f)^2\pa_2'\underline{\vu}\cdot\vN_f+\pa_1'f\pa_2'f\pa_1'\uline{\vu}\cdot\vN_f\big)}{1+(\pa_1'f)^2+(\pa_2'f)^2}.
	\end{align*}
\end{proof}

\begin{lemma}\label{lem-com-t-DN}
	For any function $g\in H^{s+1}(\bbT^2)$ with $s\ge \frac{3}{2}$, it holds that
	\begin{align*}
		|[\bar{\mathcal N}_f,D_t]g|_{H^{s}}\lesssim |f|^3_{H^{s+1}}||\vu||_{H^{s+3/2}(\Omega_f)}|g|_{H^{s+1}}.
	\end{align*}
\end{lemma}
\begin{proof}
We start to analyze of $[D_t,\bar{\mathcal H}_f]g$. Direct calculation shows that
\begin{align*}
	\Delta D_t\bar{\mathcal H}_fg&=D_t\Delta\bar{\mathcal H}_fg+2\nabla \vu:\nabla^2\bar{\mathcal H}_fg+\Delta\vu\cdot\nabla\bar{\mathcal H}_fg\\
	&=2\nabla \vu:\nabla^2\bar{\mathcal H}_fg+\Delta \vu\cdot\nabla\bar{\mathcal H}_fg.
\end{align*}
So we have
\begin{align*}
	[D_t,\bar{\mathcal H}_f]g=\bar\Delta^{-1}(2\nabla \vu:\nabla^2\bar{\mathcal H}g+\Delta \vu\cdot\nabla\bar{\mathcal H}_fgg)+\bar{\bar{\mathcal H}}_fg,
\end{align*}
where $\bar\Delta^{-1}$ means
\begin{equation}
  \left\{
  	\begin{array}{ll}
  		\Delta\bar\Delta^{-1}g=g&x\in\Omega_f;\\
  		\bar\Delta^{-1}g=0&x\in\Gamma_f;\\
  		\pa_3\bar\Delta^{-1}g=0&x\in\Gamma^-,
  	\end{array}
  \right.
\end{equation}
and $\bar{\bar{\mathcal H}}_f$ means
\begin{equation}
  \left\{
  	\begin{array}{ll}
  		\Delta\bar{\bar{\mathcal H}}_fg=0&x\in \Omega_f;\\
  		\bar{\bar{\mathcal H}}_fg=0&x\in \Gamma_f;\\
  		\pa_3\bar{\bar{\mathcal H}}_fg=\pa_3u_1\pa_1\bar{\mathcal H}_fg+\pa_3u_2\pa_2\bar{\mathcal H}_fg&x\in\Gamma^-.
  	\end{array}
  \right.
\end{equation}
Next we consider $\bar{\mathcal N}_f$. Using Lemma \ref{lem-basic-n} we  get
\begin{align*}
	[D_t,\bar{\mathcal N}_f]g&=D_t(\vN_f\cdot\uline{\nabla\bar{\mathcal H}_fg})-\vN_f\cdot\uline{\nabla\bar{\mathcal H}_f(D_tg)}\\
	&=\vN_f\cdot\uline{\nabla D_t\bar{\mathcal H}_fg}-\uline{\nabla\bar{\mathcal H}_fg}\cdot(\vN_f\cdot\uline{\nabla\vu})+\uline{\nabla\bar{\mathcal H}_fg}\cdot (D_t\vN_f)-\vN_f\cdot\uline{\nabla\bar{\mathcal H}_f(D_tg)}\\
	&=\vN_f\cdot\uline{\nabla\bar\Delta^{-1}(2\nabla \vu:\nabla^2\bar{\mathcal H}_fg}+\uline{\Delta \vu}\cdot\uline{\nabla\bar{\mathcal H}_fg)}+\vN_f\cdot\uline{\nabla\bar{\bar{\mathcal H}}_fg}-\uline{\nabla\bar{\mathcal H}_fg}\cdot(\vN_f\cdot\uline{\nabla\vu})\\
	&\quad+\frac{\big(-\pa_1'\underline{\vu}\cdot\vN_f-(\pa_2'f)^2\pa_1'\underline{\vu}\cdot\vN_f+\pa_1'f\pa_2'f\pa_2'\uline{\vu}\cdot\vN_f\big)\pa_1'g}{1+(\pa_1'f)^2+(\pa_2'f)^2}\\
	&\quad+\frac{\big(-\pa_2'\underline{\vu}\cdot\vN_f-(\pa_1'f)^2\pa_2'\underline{\vu}\cdot\vN_f+\pa_1'f\pa_2'f\pa_1'\uline{\vu}\cdot\vN_f\big)\pa_2'g}{1+(\pa_1'f)^2+(\pa_2'f)^2}\\
	&\quad+\frac{\pa_1'f\pa_1'\underline{\vu}\cdot\vN_f+\pa_2'f\pa_2'\underline{\vu}\cdot\vN_f}{1+(\pa_1'f)^2+(\pa_2'f)^2}\bar{\mathcal N}_fg.
\end{align*}
As a conclusion, it is easy for us to get
\begin{align*}
	|[\bar{\mathcal N}_f,D_t]g|_{H^{s}}\lesssim |f|^3_{H^{s+1}}||\vu||_{H^{s+3/2}(\Omega_f)}|g|_{H^{s+1}}.
\end{align*}
\end{proof}

\begin{lemma}\label{lem-com-f-DN}
	For any functions $a\in H^{3/2}(\bbT^2)$, $g\in H^{1/2}(\bbT^2)$, it holds that
	\begin{align*}
		|[\cN_f,a]g|_{L^2}\lesssim |f|_{H^3}|a|_{H^{3/2}}|g|_{H^{1/2}}.	
	\end{align*}
\end{lemma}
\begin{proof}
  Similar to Lemma \ref{lem-com-t-DN}, we have that
  \begin{align*}
    [\cN_f,a]g=g\cN_fa-2\vN_f\cdot\uline{\nabla\Delta^{-1}(\nabla \cH_fa\cdot\nabla\cH_fg)}.
  \end{align*}
It holds that
\begin{align*}
  |g\cN_fa|_{L^2}\le|g|_{L^4}^{1/2}|\cN_fa|_{L^4}^{1/2}\lesssim|g|_{H^{1/2}}^{1/2}|\cN_fa|_{H^{1/2}}^{1/2}.
\end{align*}
On the other hand, one has
\begin{align*}
  |\vN_f\cdot\uline{\nabla\Delta^{-1}(\nabla \cH_fa\cdot\nabla\cH_fg)}|_{L^2}\lesssim |f|_{H^{5/2}}||\nabla \cH_fa\cdot\nabla\cH_fg||_{H^{-1/2}(\Om_f)}.
\end{align*}
For each test function $\phi$ with $||\phi||_{H^{1/2}(\Om_f)}=1$, we have
\begin{align*}
  &\int_{\Omega_f}\nabla \cH_fa\cdot\nabla\cH_fg\phi dx\\
  \le&||\phi||_{L^3(\Om_f)}||\nabla\cH_fg||_{L^2(\Om_f)}||\nabla \cH_fa||_{L^6(\Om_f)}\\
  \lesssim&||\phi||_{H^{1/2}(\Om_f)}|g|_{H^{1/2}}|a|_{H^{3/2}},
\end{align*}
then the conclusion follows easily.
\end{proof}

\begin{lemma}\label{lem-DN-b}
	For any functions $a \in H^{3/2}(\bbT^2)$, $g\in H^{1/2}(\bbT^2)$, it holds that
	\begin{align*}
		\int_{\bbT^2} a(\cN_fg)gdx'\lesssim |f|_{H^3}|a|_{H^{3/2}}|g|^2_{H^{1/2}}.
	\end{align*}
\end{lemma}

\begin{proof}
We use $\tilde a=\cH_fa$ to denote the harmonic extension of $a$. Then it holds that
\begin{align*}
	\int_{\bbT^2}a(\cN_fg)gdx'
	=&\int_{\Omega_f}(\nabla \cH_fg)\cdot\nabla(\tilde a\cH_fg)dx\\
	=&\int_{\Omega_f}(\nabla \cH_fg)\cdot(\nabla \tilde a\cH_fg+\tilde a\nabla \cH_fg)dx.
\end{align*}
Using integration by parts, we know that
\begin{align*}
	\int_{\Omega_f}\nabla \cH_fg\cdot(\nabla \tilde a\cH_fg)dx
	=&\int_{\bbT^2}\vN_f\cdot\uline{\nabla \tilde a} ggdx'-\int_{\Omega_f}\cH_fg\nabla\tilde a\cdot\nabla\cH_fgdx.
\end{align*}
Thus we have
\begin{align*}
	\int_{\bbT^2}a\cN_fggdx'
	=&\int_{\Omega_f}\tilde a(\nabla \cH_fg)^2dx+\frac{1}{2}\int_{\bbT^2}\cN_f a ggdx'\\
	\lesssim&|a|_{L^\infty}|g|^2_{H^{1/2}}+|f|_{H^3}|a|_{H^1}|g|^2_{L^4}\\
	\lesssim&|f|_{H^3}|a|_{H^{3/2}}|g|^2_{H^{1/2}}.
\end{align*}
\end{proof}
\begin{lemma}\label{lem-com-t-DN-2}
For any function $a\in H^{5/2}(\bbT^2)$, $g\in H^{1/2}(\bbT^2)$, it holds that
	\begin{align*}
		\int_{\bbT^2}a([D_t,\cN_f]g)gdx'\le C |a|_{H^{5/2}}|f|^2_{H^4}||\vu||_{H^{4}(\Omega_f)}|g|^2_{H^{1/2}}.
	\end{align*}
	
\end{lemma}
\begin{proof}
	Similar to Lemma \ref{lem-com-t-DN}, we have
\begin{align*}
	[D_t,\cN_f]g=&\vN_f\cdot\uline{\nabla\Delta^{-1}(2\nabla \vu:\nabla^2{\mathcal H}_fg+\Delta \vu\cdot\nabla{\mathcal H}_fg)}-\uline{\nabla{\mathcal H}_fg}\cdot(\vN_f\cdot\uline{\nabla\vu})\\
	&+\frac{\big(-\pa_1'\underline{\vu}\cdot\vN_f-(\pa_2'f)^2\pa_1'\underline{\vu}\cdot\vN_f+\pa_1'f\pa_2'f\pa_2'\uline{\vu}\cdot\vN_f\big)\pa_1'g}{1+(\pa_1'f)^2+(\pa_2'f)^2}\\
	&+\frac{\big(-\pa_2'\underline{\vu}\cdot\vN_f-(\pa_1'f)^2\pa_2'\underline{\vu}\cdot\vN_f+\pa_1'f\pa_2'f\pa_1'\uline{\vu}\cdot\vN_f\big)\pa_2'g}{1+(\pa_1'f)^2+(\pa_2'f)^2}\\
	&+\frac{\pa_1'f\pa_1'\underline{\vu}\cdot\vN_f+\pa_2'f\pa_2'\underline{\vu}\cdot\vN_f}{1+(\pa_1'f)^2+(\pa_2'f)^2}{\mathcal N}_fg,
\end{align*}
where $\Delta^{-1}$ means
\begin{equation}
  \left\{
  	\begin{array}{ll}
  		\Delta\Delta^{-1}g=g&x\in\Omega_f;\\
  		\Delta^{-1}g=0&x\in\Gamma_f;\\
  		\Delta^{-1}g=0&x\in\Gamma^-.
  	\end{array}
  \right.
\end{equation}
Using  $\tilde a=\cH_fa$ to denote the harmonic extension of $a$,  we can write
\begin{align*}
	&\int_{\bbT^2}a\vN_f\cdot\uline{\nabla\Delta^{-1}(2\nabla \vu:\nabla^2{\mathcal H}_fg+\Delta \vu\cdot\nabla{\mathcal H}_fg)}gdx'\\
	=&\int_{\Omega_f}(2\nabla \vu:\nabla^2{\mathcal H}_fg+\Delta \vu\cdot\nabla{\mathcal H}_fg)\tilde a \cH_fgdx\\
	&+\int_{\Omega_f}\nabla\Delta^{-1}(2\nabla \vu:\nabla^2{\mathcal H}_fg+\Delta \vu\cdot\nabla{\mathcal H}_fg)\cdot(\tilde a \nabla \cH_fg+\cH_fg\nabla\tilde a)dx.
\end{align*}
Since $\vu$ is divergence free, one can get
\begin{align*}
	&\int_{\Omega_f}(2\nabla \vu:\nabla^2{\mathcal H}_fg)\tilde a \cH_fgdx\\
	=&2\int_{\Omega_f}(\pa_ru_s\pa_r\pa_s{\mathcal H}_fg)\tilde a \cH_fgdx\\
	=&\int_{\Omega_f}(\pa_ru_s\pa_r\pa_s{\mathcal H}_fg)\tilde a \cH_fgdx+\int_{\bbT^2}(\uline{\nabla{\mathcal H}_fg}\cdot(\vN_f\cdot\uline{\nabla\vu}))a gdx'\\
	&-\int_{\Omega_f}(\pa^2_ru_s\pa_s{\mathcal H}_fg)\tilde a \cH_fgdx-\int_{\Omega_f}\pa_ru_s\pa_s{\mathcal H}_fg\pa_r(\tilde a \cH_fg)dx\\
	=&\int_{\bbT^2}(\uline{\nabla{\mathcal H}_fg}\cdot(\vN_f\cdot\uline{\nabla\vu}))a gdx'-\int_{\Omega_f}(\pa^2_ru_s\pa_s{\mathcal H}_fg)\tilde a \cH_fgdx\\
	&+\int_{\Omega_f}\pa_ru_s\big(\pa_s(\pa_r{\mathcal H}_fg\tilde a \cH_fg)-\pa_r{\mathcal H}_fg\pa_s(\tilde a \cH_fg)-\pa_s{\mathcal H}_fg\pa_r(\tilde a \cH_fg))dx\\
	=&\int_{\bbT^2}(\uline{\nabla{\mathcal H}_fg}\cdot(\vN_f\cdot\uline{\nabla\vu}))a gdx'-\int_{\Omega_f}(\pa^2_ru_s\pa_s{\mathcal H}_fg)\tilde a \cH_fgdx\\
	&+\int_{\bbT^2}(\uline{\pa_r\vu}\cdot\vN_f\uline{\pa_r{\mathcal H}_fg})a gdx'-\int_{\Omega_f}\pa_ru_s\big(\pa_r{\mathcal H}_fg\pa_s(\tilde a \cH_fg)+\pa_s{\mathcal H}_fg\pa_r(\tilde a \cH_fg))dx.
\end{align*}
Analyzing the boundary term carefully, we find that
\begin{align*}
	&\int_{\bbT^2}(\uline{\pa_r\vu}\cdot\vN_f\uline{\pa_r{\mathcal H}_fg)}a gdx'\\
	=&\int_{\bbT^2}\frac{\big(\pa_1'\underline{\vu}\cdot\vN_f+(\pa_2'f)^2\pa_1'\underline{\vu}\cdot\vN_f-\pa_1'f\pa_2'f\pa_2'\uline{\vu}\cdot\vN_f\big)\pa_1'g}{1+(\pa_1'f)^2+(\pa_2'f)^2}a gdx'\\
	&+\int_{\bbT^2}\frac{\big(\pa_2'\underline{\vu}\cdot\vN_f+(\pa_1'f)^2\pa_2'\underline{\vu}\cdot\vN_f-\pa_1'f\pa_2'f\pa_1'\uline{\vu}\cdot\vN_f\big)\pa_2'g}{1+(\pa_1'f)^2+(\pa_2'f)^2}a gdx'\\
	&+\int_{\bbT^2}\frac{\big(\uline{(\vN_f\cdot\nabla)\vu}\cdot\vN_f\big)\cN_fg}{1+(\pa_1'f)^2+(\pa_2'f)^2}a gdx',
\end{align*}
Combining above results, we have
\begin{align*}
	&\int_{\bbT^2}a([D_t,\cN_f]g)gdx'\\
	=&\int_{\Omega_f}\nabla\Delta^{-1}(2\nabla \vu:\nabla^2{\mathcal H}_fg+\Delta \vu\cdot\nabla{\mathcal H}_fg)(\tilde a \nabla g+g\nabla\tilde a)dx\\
	&+\int_{\bbT^2}\uline{\pa_3\vu}\cdot\vN_f\cN_fga gdx'-\int_{\Omega_f}(\nabla\vu+\nabla\vu^\top):\nabla\cH_f g\otimes\nabla(\tilde a\cH_fg)dx.
\end{align*}
Using Lemma \ref{lem-DN-b}, we finally arrive at
\begin{align*}
	\int_{\bbT^2}a([D_t,\cN_f]g)gdx'\le C |a|_{H^{5/2}}|f|^2_{H^4}||\vu||_{H^{4}(\Omega_f)}|g|^2_{H^{1/2}}.
\end{align*}
\end{proof}
\end{appendix}

\section*{Acknowledgment}
Wei Wang is supported by NSF of China under Grant 11871424 and 11922118 and and the Young Elite Scientists
Sponsorship Program by CAST. Zhifei Zhang is partially supported by NSF of China under Grant 11425103.


\begin{thebibliography}{99}
 \bibitem{ABZ} T. Alazard, N. Burq and C. Zuily, {\it On the Cauchy problem for gravity water waves},
 Invent. Math.,  198(2014), 71-163.
 
\bibitem{Ax} W. I. Axford, {\it Note on a problem of magnetohydrodynamic stability}, Canad. J. Phys., 40(1962), 654-655.

\bibitem{AM}  D.  M. Ambrose and N. Masmoudi, {\it Well-posedness of 3D vortex sheets with surface tension}, Commun. Math. Sci.,  5 (2007) 391-430.

\bibitem{CCS} C. A. Cheng, D. Coutand and S. Shkoller, {\it On the motion of vortex sheets with surface tension in three-dimensional Euler equations with vorticity}, Comm. Pure Appl. Math., 61 (2008), 1715-1752.

\bibitem{CHW} R. M. Chen, J. Hu and D. Wang, {\it Linear stability of compressible vortex sheets in
two-dimensional elastodynamics}, Advances in Mathematics, 311 (2017), 18-60.

\bibitem{CHWY} R. M. Chen, J. Huang, D. Wang and D. Yuan, {\it Stabilization effect of elasticity on three-dimensional compressible vortex sheet}, preprint.

\bibitem{Chen} G.-Q. Chen and Y.-G Wang, {\it Existence and stability of compressible current-vortex sheets in three-dimensional magnetohydrodynamics},  Arch. Ration. Mech. Anal., 187(2008), 369-408.
 
 \bibitem{CMST} J.-F. Coulombel, A. Morando, P.  Secchi and P. Trebeschi, {\it A priori estimates for 3D incompressible current-vortex sheets}, Comm. Math. Phys.,  311(2012), 247-275.
              
\bibitem{CL}D. Christodoulou and H. Lindblad, {\it On the motion of the free surface of a liquid} , Comm. Pure Appl. Math.,  53(2000),1536-1602.

\bibitem{Eb}  D. Ebin, {\it he equations of motion of a perfect fluid with free boundary are not well posed}, Comm. Partial Differential Equations, 12(1987),1175-1201.

\bibitem{GW} X. Gu and F. Wang, {\it Well-posedness of the free boundary problem in incompressible elastodynamics under the mixed type stability condition},  J. Math. Anal. Appl., 482 (2020), 123529, 30 pp.

\bibitem{Hao} C. Hao, {\it On the motion of free interface in ideal incompressible MHD}, Arch. Ration. Mech. Anal.,  224 (2017), 515-553. 

\bibitem{HL} C. Hao and T. Luo, {\it A priori estimates for free boundary problem of incompressible inviscid magnetohydrodynamic flows}, Arch. Ration. Mech. Anal.,  212(2014), 805-847.
             
\bibitem{HW} C. Hao and D. Wang, {\it A priori estimates for the free boundary problem of incompressible neo-Hookean elastodynamics}, J.Differential Equations, 261(2016), 712-737.

\bibitem{La} D. Lannes, {\it Well-posedness of the water-waves equations}, J.  Amer.  Math.  Soc., 18(2005), 605-654.

\bibitem{LWZ} H. Li, W. Wang and  Z. Zhang, {\it Well-posedness of the free boundary problem in incompressible elastodynamics},  J. Differential Equations, 267 (2019),  6604-6643.


\bibitem{Maj} A. Majda and A.  Bertozzi, {\it Vorticity and incompressible flow}, Cambridge Texts in Applied Mathematics, 27, Cambridge University Press, Cambridge, 2002.
\bibitem{MTT1} A. Morando, Y. Trakhinin and P. Trebeschi,  {\it Stability of incompressible current-vortex sheets},
J. Math. Anal. Appl., 347(2008), 502-520.

\bibitem{MTT2}  A. Morando, Y. Trakhinin and P. Trebeschi,  {\it Well-posedness of the linearized plasma-vacuum interface problem in ideal incompressible MHD}, Quart. Appl. Math., 72(2014), 549-587.
\bibitem{ST} P. Secchi and Y. Trakhinin, {\it Well-posedness of the plasma-vacuum interface problem},  Nonlinearity, 27(2014), 105-169.

\bibitem{SWZ} Y. Sun, W. Wang and Z. Zhang, {\it Nonlinear stability of the current-vortex sheet to the incompressible MHD equations}, Comm. Pure Appl. Math., 71(2018), 356-403.

\bibitem{SWZ2} Y. Sun, W. Wang and Z. Zhang, {\it Well-posedness of the plasma-vacuum interface problem for ideal incompressible MHD}, Arch. Ration. Mech. Anal.,  234 (2019),  81-113. 


\bibitem{Sy} S. I. Syrovatskij, {\it The stability of tangential discontinuities in a magnetohydrodynamic medium},
 Z. Eksperim. Teoret. Fiz., 24(1953), 622-629.
 
 \bibitem{SZ1} J. ~Shatah and C. Zeng,  {\it Geometry and  a priori estimates for free boundary problems of the Euler's equation}, Comm. Pure Appl. Math., 61(2008), 698-744.
 
\bibitem{SZ2} J. ~Shatah and C. Zeng,  {\it A priori estimates for fluid interface problems}, Comm. Pure Appl. Math., 61(2008), 848-876.

\bibitem{Tra1} Y. Trakhinin, {\it Existence of compressible current-vortex sheets: variable coefficients linear analysis}, Arch. Ration. Mech. Anal., 177(2005), 331-366.

\bibitem{Tra2} Y.  Trakhinin,  {\it The existence of current-vortex sheets in ideal compressible magnetohydrodynamics}, Arch. Ration. Mech. Anal., 191(2009), 245-310.

\bibitem{Tra-in} Y. Trakhinin, {\it On the existence of incompressible current-vortex sheets: study of a linearized free boundary value problem}, Math. Methods Appl. Sci.,  28(2005), 917-945.
           
\bibitem{Tra-JDE} Y.  Trakhinin, {\it On the well-posedness of a linearized plasma-vacuum interface problem in ideal compressible MHD}, { J. Differential Equations}, {249}(2010), 2577-2599.

\bibitem{Tra3} Y. Trakhinin, {\it Well-posedness of the free boundary problem in compressible elastodynamics},  
J. Differential Equations, 264 (2018), 1661-1715. 

\bibitem{Wu1} S. Wu, {\it Well-posedness in Sobolev spaces of the full water wave problem in $2$-D},  Invent. Math., 130(1997), 39--72.

\bibitem{Wu2} S. Wu, {\it Well-posedness in Sobolev spaces of the full water wave problem in 3-D}, J. ~Amer. ~Math. ~Soc., 12(1999), 445-495.

\bibitem{WY}  Y.-G. Wang and F. Yu,  {\it Stabilization effect of magnetic fields on two-dimensional compressible current-vortex sheets}, Arch. Ration. Mech. Anal.,  208(2013), 341-389.

\bibitem{ZZ} P. Zhang and Z. Zhang, {\it On the free boundary problem of  three-dimensional incompressible Euler equations},
Comm. Pure Appl. Math., 61(2008), 877--940.
\end{thebibliography}
\end{document}